\documentclass{amsart}
\usepackage{amsmath}
\usepackage{amsthm}
\usepackage{amssymb}
\usepackage{amsfonts}
\usepackage{amsxtra}     
\usepackage{epsfig}
\usepackage{graphicx}
\usepackage{graphics}

\newtheorem{theorem}{Theorem}[section]

\newtheorem{corollary}[theorem]{Corollary}
\newtheorem{lemma}[theorem]{Lemma}
\newtheorem{sublemma}[theorem]{Sublemma}

\newtheorem*{andreev}{Andreev's Theorem}

\DeclareMathOperator{\link}{link}
\DeclareMathOperator{\vol}{vol}

\theoremstyle{remark}

\newcommand{\HH}{{\mathbb H}}
\newcommand{\bS}{{\mathbb S}}
\begin{document}

\title{Organizing Volumes of Right-Angled Hyperbolic Polyhedra}
\author{Taiyo Inoue}
\address{Department of Mathematics, University of California, Berkeley}
\email{inoue@math.berkeley.edu}

\begin{abstract} This article defines a pair of combinatorial operations on the combinatorial structure of compact right-angled hyperbolic polyhedra in dimension three called decomposition and edge surgery.  It is shown that these operations simplify the combinatorics of such a polyhedron, while keeping it within the class of right-angled objects, until it is a disjoint union of L\"obell polyhedra, a class of polyhedra which generalizes the dodecahedron.  Furthermore, these combinatorial operations are shown to have geometric realizations which are volume decreasing.  This allows for an organization of the volumes of right-angled hyperbolic polyhedra and allows, in particular, the determination of the polyhedra with smallest and second smallest volumes.  
\end{abstract}

\maketitle

\section{Introduction}

Three-dimensional right-angled hyperbolic compact polyhedra are a well understood family of hyperbolic objects.  A complete classification of them is provided by Andreev's Theorem \cite{And} \cite{RHD}, which characterizes when a combinatorial polyhedron admits a geometric realization in $\HH^3$ with a given non-obtuse dihedral angle specified at each edge.  A simple corollary of this result is a complete characterization of combinatorial types of right-angled hyperbolic polyhedra by means of a small set of conditions on the combinatorics.

Nonetheless, the problem of determining the volume of a given right-angled hyperbolic polyhedra remains difficult.  This article will attempt to organize these volumes.  

By Mostow rigidity, the hyperbolic structure of a three dimensional right-angled hyperbolic polyhedron is unique and is determined by 
the combinatorial structure of the polyhedron.  In particular, volume is a combinatorial invariant for those polyhedra which admit a right-angled hyperbolic realization.  

Following this spirit, volumes of these polyhedra will be organized according to a combinatorial 
process which, at each step, reduces the complexity of the polyhedron.  The 
two combinatorial operations used in this process are called decomposition and edge surgery.  The former is a splitting of the polyhedron under certain conditions, while the latter is the deletion of a particular edge.  These operations, when applied to sufficiently 
complicated polyhedra, reduce the complexity of the polyhedron until, after enough applications, the resulting object is a finite disjoint union of polyhedra from an infinite family of exceptional polyhedra, namely the L\"obell polyhedra.  This family of polyhedra have the property that these operations cannot be applied to them and their geometry, in particular their volume, is very well understood.

It will be shown that for polyhedra not of L\"obell type, either the polyhedron is decomposable or there exists an edge for which edge surgery can be performed.  Therefore, every such polyhedron will eventually be transformed into a family of L\"obell ones via an iteration of this process.

These combinatorial operations will then be studied from a geometric point of view.  Decomposition bears a strong resemblance to the decomposition of hyperbolic Haken manifolds along incompressible subsurfaces.  In particular, after passing to a manifold cover of a particular sort constructed in Section 5, this is precisely what decomposition is.  As such, the result of Agol, Dunfield, Storm and Thurston \cite{ADST}, which gives a description of the effect of Haken decomposition on volumes, can be applied to show that decomposition of right-angled hyperbolic polyhedra is not volume increasing.

Next, it will be shown that the geometric realization of edge surgery is to ``unbend'' the polyhedron along the edge that is surgered.   This means that the polyhedron is deformed so that the dihedral angle measure along this edge increases from $\frac{\pi}{2}$ to $\pi$ while the dihedral angle measure of every other edge is constant and equal to $\frac{\pi}{2}$.  The main sticking point with this is that this deformation passes through obtuse-angled polyhedra which is not the purview of Andreev's theorem.  However, results of Rivin and Hodgson \cite{RH} generalizing Andreev's theorem will imply that this deformation exists.  Then by means of the Schl\"afli differential formula, it will be shown that the geometric realization of this operation is volume decreasing.  

Putting these things together gives chains of inequalities of volumes of right-angled hyperbolic polyhedra determined by the decompositions and edge surgeries used to go from an initial polyhedron to a disjoint union of L\"obell polyhedra.  Since every right-angled hyperbolic 
polyhedron which is not of L\"obell type can be reduced in this way, a method for organizing volumes of right-angled hyperbolic polyhedra is obtained.  This is summarized in the following theorem, which is the main result of this article: 

\smallskip

\noindent{\bf Theorem \ref{organization}.} {\sl Let $P_0$ be a compact right-angled hyperbolic polyhedron.  Then there exists a sequence of disjoint unions of right-angled hyperbolic polyhedra $P_1, P_2, \dots, P_k$ such that for $i=1,\dots, k$, $P_i$ is gotten from $P_{i-1}$ by either a decomposition or edge surgery, and $P_k$ is a set of L\"obell polyhedra. Furthermore, 
$$\vol(P_0) \geq \vol(P_1) \geq \vol(P_2) \geq \dots \geq \vol(P_k).$$
}
As the volumes of L\"obell polyhedra can be explicitly calculated, the right-angled hyperbolic polyhedra of smallest and second smallest volumes can be identified easily:

\smallskip

\noindent{\bf Corollary \ref{smallestvol}.} {\sl The compact right-angled hyperbolic polyhedron of smallest volume is $L(5)$ (a dodecahedron) and the second smallest is $L(6)$ where $L(n)$ denotes the nth L\"obell polyhedron.}

\smallskip

In fact, if one had an oracle to inform them about the precise volume of a given polyhedron, then this result would provide an algorithm terminating in finite time to determine the ordering of volumes in the sense that the right-angled hyperbolic polyhedron of $n$th smallest volume would be identifiable.

It should be noted that this article deals exclusively with compact right-angled hyperbolic polyhedra.  Their ideal siblings, those right-angled hyperbolic polyhedra which have vertices lying on the ideal boundary of $\HH^3$, are not covered.  There are examples of ideal right-angled polyhedra whose volumes are smaller than that of the smallest compact right-angled hyperbolic polyhedron.  For example, the right-angled ideal octahedron, all of whose vertices are ideal, has volume strictly smaller than that of the right-angled dodecahedron.

Furthermore, at the present moment, analogous questions about four dimensional and higher dimensional right-angled hyperbolic polyhedra are largely a mystery and cannot be dealt with using the techniques of this article. Indeed, there is no technology which rivals the power of Andreev's theorem to even begin constructing examples of such objects in dimension 4 and higher.  



\smallskip

\noindent {\bf Acknowledgments:} This article is essentially identical to my Ph.D dissertation supervised by Rob Kirby and Alan Reid filed at U.C. Berkeley in 2007.  I am grateful to them for inspiration, knowledge and support.  I thank the referee for her/his comments as well, as well as the journal's editors.

\section{Definitions and Notations}

The setting for this article is ${\mathbb H}^3$.  More generally, ${\HH^n}$ is the unique simply connected Riemannian manifold of dimension $n$ with constant sectional curvatures equal to $-1$.  

If $F$ is a totally geodesic $2$--plane in ${\mathbb H}^3$ then a {\it hyperbolic half-space} $H_F$ is a closed subset of $\HH^3$ bounded by $H$.  Similarly, if $g$ is a geodesic in $\HH^2$, then a {\it hyperbolic half-plane} is a closed subset of $\HH^2$ bounded by $g$.   

A {\it hyperbolic polyhedron} is a nonempty compact convex transverse intersection of a finite number of hyperbolic half-spaces. For any hyperbolic polyhedron $P$, there will be a unique minimal set of hyperbolic half spaces whose intersection is $P$.  It will be assumed that this minimal set will always be the one defining $P$.  

If $H_F$ is a hyperbolic half-space defining $P$, a {\it face} of $P$ is the intersection of $F$ with $P$.  The $2$--plane $F$ is said to be the plane {\it supporting} the face.  As $F$ is itself isometric to the hyperbolic plane, it is evident that a face is itself isometric to a {\it hyperbolic polygon}, that is a nonempty convex compact transverse intersection of a finite number of hyperbolic half-planes.  By abusing notation, the face supported by $F$ will often be denoted by the same letter $F$.  

If $P$ is a hyperbolic polyhedron, then an {\it edge} is a nonempty intersection of two distinct faces of $P$ containing more than one point. A {\it vertex} is a nonempty intersection of three or more distinct faces of $P$.

A {\it combinatorial polyhedron} is a 3--ball whose boundary sphere is equipped with a cell structure whose 0--cells, 1--cells and 2--cells will also be called vertices, edges and faces respectively, and which can be realized as a convex polyhedron.  By Steinitz's theorem, such objects are exactly those whose 1--skeletons are simple and 3--connected graphs.  A hyperbolic polyhedron has a natural description as a combinatorial polyhedron.  Passage between the combinatorial perspective and the geometric one will often be done without mention.

Let $c$ be a simple closed curve on $\partial P$ which intersects transversely the interior of exactly $k$ distinct edges.  Such a curve is called a {\it k--circuit}.  A $k$--circuit is a {\it prismatic k--circuit} if the endpoints of all the edges which $c$ intersects are distinct.  Often the distinction between a $k$--circuit $c$ and the edges it intersects will be blurry, if not completely nonexistent.  See Figure \ref{fig: pc} for examples of prismatic circuits in combinatorial polyhedra.

\begin{figure}
	\centering
	\includegraphics[bb=14 14 271 231]{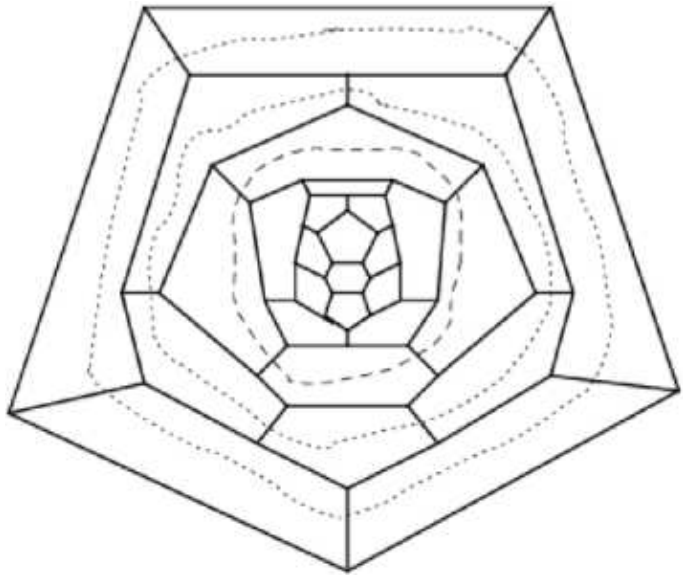}
	\caption{This is the 1--skeleton of a polyhedron with prismatic circuits.  The dashed curve is a prismatic 4--circuit and the dotted curves are prismatic 5--circuits.}
	\label{fig: pc}
\end{figure}

Let $e$ be an edge of $P$ given by the intersection of a pair of distinct faces $F$ and $G$ which are supported by planes also denoted $F$ and $G$.  Then the {\it interior dihedral angle} (often simply the {\it dihedral angle}) at $e$ is the dihedral angle formed by the planes $F$ and $G$ in the interior of $P$.  Note that because $P$ is a convex set, the interior dihedral angle of an edge always has measure strictly smaller than $\pi$.  The {\it exterior dihedral angle} at $e$ is the dihedral angle formed by $F$ and $G$ in the exterior of $P$ which is the supplement of the interior dihedral angle.


\section{Right-Angled Hyperbolic Polyhedra}

The primary objects of study in this article are {\it right-angled hyperbolic polyhedra} which are hyperbolic polyhedra all of whose dihedral angles have measure equal to $\frac{\pi}{2}$.  A hyperbolic polyhedron all of whose dihedral angles have measure less than or equal to $\frac{\pi}{2}$ is said to be {\it non-obtuse}.  When one restricts attention to this class, there are more constraints placed on the combinatorics of the polyhedron than just Steinitz's Theorem.  A set of necessary and sufficient conditions for the existence of a non-obtuse hyperbolic polyhedron is described by Andreev's Theorem \cite{And}:




  
\begin{andreev} A combinatorial polyhedron $P$ which is not isomorphic to a tetrahedron or a triangular prism has a geometric realization in $\HH^3$ with interior dihedral angle measures $0 < \theta_i \leq \frac{\pi}{2}$ at edge $e_i$ if and only if:

\setcounter{enumi}{0}
\begin{enumerate}

\item The $1$--skeleton of $P$ is trivalent.

\item If $e_i, e_j, e_k$ are distinct edges which meet at a vertex, then $\theta_i + \theta_j + \theta_k > \pi$.

\item If $e_i, e_j, e_k$ form a prismatic $3$--circuit, then $\theta_i + \theta_j + \theta_k < \pi$.

\item If $e_i, e_j, e_k, e_l$ form a prismatic $4$--circuit, then $\theta_i + \theta_j + \theta_k + \theta_l < 2\pi$.

\end{enumerate}

This geometric realization is unique up to isometry of $\HH^3$. 
\end{andreev}

Andreev's original statement and proof of this result in 1970 \cite{And} contains a flaw in one of its combinatorial arguments.  A corrected proof is presented in Roeder, Hubbard and Dunbar's paper \cite{RHD}.  Andreev's theorem was also proved independently as a consequence of Rivin and Hodgson's generalization \cite{RH}.  Details of this can be found in Hodgson's article \cite{H}.

The prismatic circuit conditions in Andreev's Theorem can be thought of as a cone manifold translation of the concepts ``irreducible'' and ``atoroidal''.  For example, a prismatic $3$--circuit implies the existence of a topologically embedded triangle in $P$ whose vertices lie on the edges of the circuit.  If the dihedral angle measures of these edges add up to more $\pi$, then geometrically this triangle is positively curved and so can be thought of as the analogue of an embedded $2$--sphere in a three manifold.  That the $3$--circuit is prismatic is what implies the ``incompressibility'' of the triangle. If the dihedral angle measures add up to $\pi$ exactly, then the embedded triangle is flat, and so can be thought of as the analogue of an incompressible torus.  A similar discussion works for prismatic $4$--circuits.  Similarly, the second condition is a cone manifold version of the ``spherical link'' condition for manifolds.

Note that for prismatic $k$--circuits with $k \geq 5$, an embedded polygon whose vertices lie on these edges is geometrically negatively curved if the polyhedron is non-obtuse.  However, this might fail to be the case without the non-obtuse condition.  Thus, the classification
of hyperbolic polyhedra without the non-obtuse restriction is necessarily more subtle.  I Rivin and C Hodgson have accomplished this generalization which will be discussed later.

By requiring that all dihedral angle measures be $\frac{\pi}{2}$, Andreev's classification of right-angled hyperbolic polyhedra becomes purely combinatorial.  In fact, the following classification of right-angled hyperbolic polyhedra was done by A Pogorelov in 1967 \cite{Pog} before Andreev's work:

\begin{corollary} \label{rightclassified} A combinatorial polyhedron $P$ has a geometric realization in $\HH^3$ as a right-angled hyperbolic polyhedron if and only if:

\setcounter{enumi}{0}
\begin{enumerate}

\item The $1$--skeleton of $P$ is trivalent.

\item There are no prismatic $3$ or $4$--circuits.
\end{enumerate}
This geometric realization is unique up to isometry.

\end{corollary}

The geometry of a face of a right-angled polyhedron is described by the following theorem:

\begin{theorem} If $P$ is a right-angled hyperbolic polyhedron, then all of its faces are right-angled hyperbolic polygons.
\end{theorem}
\begin{proof} Let $F$ be a face of $P$ and $v$ a vertex of $F$.  Then $\link(v)$ is a right-angled spherical triangle.  Such a triangle is unique up to isometry on the round unit sphere and has edge lengths $\frac{\pi}{2}$.  These edge lengths are precisely the angles in the faces containing $v$ at $v$.  Therefore, in particular, the angle of $F$ at $v$ is $\frac{\pi}{2}$.  
\end{proof}

This fact, combined with the well known classification of right-angled hyperbolic polygons which states, in particular, that right-angled hyperbolic $k$--gons exist for $k \geq 5$, (see, for example, \cite{CM}) implies that each face of a right-angled hyperbolic polyhedron has at least 5 edges.  However, it should be noted that requiring all faces to have at least five edges is not a sufficient replacement for the prismatic circuit conditions in Andreev's characterization of right-angled polyhedra (Corollary \ref{rightclassified}).  Figure \ref{fig: pc} shows an example of a combinatorial polyhedron with trivalent 1--skeleton, all of whose faces have at least five edges, but which possesses a prismatic 4--circuit and therefore cannot be realized as a right-angled hyperbolic polyhedron.

It should also be remarked that the above result generalizes to higher dimensions.  That is, if $P$ is a right-angled polytope in $\HH^n$, then any lower dimensional face of $P$ is isometric to a right-angled hyperbolic polytope of the appropriate dimension.  However, it has been shown that compact right-angled hyperbolic polytopes in $\HH^n$ only exist for $n \leq 4$ \cite{AV} (although higher dimensional \emph{ideal} right-angled hyperbolic polytopes do exist).

Here are some combinatorial consequences of the above conditions on the structure of right-angled hyperbolic polyhedra which will be of use later:

\begin{lemma}\label{twelvepentagons} Every right-angled hyperbolic polyhedron $P$ has at least 12 pentagonal faces.  If a right-angled hyperbolic polyhedron contains only pentagonal faces, then it is a dodecahedron.  In particular, the dodecahedron is the right-angled hyperbolic polyhedron with the smallest number of faces.
\end{lemma}

\begin{proof} Let $F(P)$ denote the set of faces of $P$ and $E(F)$ denote the number of edges which the face $F$ contains.  Let $v$, $e$ and $f$ denote the number of vertices, edges and faces of $P$.  Let $$c(P) = \sum_{F \in F(P)} E(F) - 5.$$  Note that $c(P)$ is also equal to $2e - 5f$.

By Euler's formula, $v - e + f = 2$.  By trivalence of the 1--skeleton of $P$, $e = \frac{3v}{2}$.  Therefore $f - \frac{e}{3} =  2$, or $f - 2e + 5f = 12$, so $f - c(P) = 12$.

Let $k$ denote the number of pentagonal faces of $P$.  Then $c(P) \geq f-k$ as each face of $P$ which is not a pentagon contributes at least 1 to $c(P)$.  Therefore, $12 = f - c(P) \leq k$ which proves the first claim.

Suppose $P$ contains only pentagonal faces.  Then $c(P) = 0$, so $f = 12$.  It is easy to see that the only polyhedron which has 12 faces which are all pentagons is the dodecahedron.
\end{proof}

\begin{lemma}\label{verteximpliesadjacent} If $A$ and $B$ are a pair of distinct faces in a right-angled hyperbolic polyhedron $P$ which share a vertex, then they are adjacent in an edge which contains this vertex.
\end{lemma}

\begin{proof} Suppose $A$ is not adjacent to $B$ in an edge containing the vertex they share.  Then there are at least 4 edges emanating from this vertex which contradicts the trivalence of $P$.  
\end{proof}

\begin{lemma}\label{threeadj} If $A$, $B$, and $C$ are pairwise adjacent faces of right-angled hyperbolic polyhedron $P$, then they all share a vertex.
\end{lemma}

\begin{proof} Let $a = A \cap B$, $b = B\cap C$, $c = C \cap A$.  Suppose no two of $a$, $b$ and $c$ share an endpoint.  Then these three edges form a prismatic 3--circuit which is a contradiction.  

Suppose then that two of $a,b,c$ share an endpoint $v$, for example $a$ and $b$.  Then $A$ and $C$ share the vertex $v$ which implies that they are adjacent in an edge which contains this vertex.  This edge must be $c$ since a pair of faces in a polyhedron can intersect in at most one edge.  Therefore, $A$, $B$ and $C$ share a vertex.
\end{proof}

\begin{lemma}\label{fouradj} Suppose $P$ is a right-angled hyperbolic polyhedron and $A$ and $C$ are nonadjacent faces both adjacent to a face $B$.  If $D \neq B$ is also adjacent to both $A$ and $C$, then $D$ is also adjacent to $B$.
\end{lemma}

\begin{proof} Suppose for a contradiction that $D$ is not adjacent to $B$.  Consider the cycle of faces $A$, $B$, $C$, $D$.  Let $e_1 = A \cap B$, $e_2 = B \cap C$, $e_3 = C \cap D$, and $e_4 = D \cap A$.  

Suppose $e_1$ and $e_2$ share an endpoint.  Then $A$ and $C$ share a vertex and therefore, by the above proposition, are adjacent which is a contradiction.  A similar argument shows $e_i$ and $e_j$ share no endpoints for $i \neq j$.  

Therefore, these edges form a prismatic 4--circuit which is a contradiction.  \end{proof}

\section{Examples of Right-Angled Hyperbolic Polyhedra}

This section will describe an important family of right-angled hyperbolic polyhedra, and will describe operations for producing new examples from given ones.

\smallskip

\noindent {\bf L\"obell Polyhedra:} A {\it pentagonal flower}, denoted $L^n$ for $n\geq 5$, is a combinatorial $2$--complex consisting of an $n$--gon $F$ surrounded by $n$ pentagons $p_1, \dots, p_n$ with indices ordered cyclically such that $p_i$ is adjacent to $F$, $p_{i-1}$ and $p_{i+1}$.  In the case when $n = 5$, a pentagonal flower $L^5$ is called a {\it dodecahedral flower}.
 
Let $L_1^n$ and $L_2^n$ be a pair of pentagonal flowers.  Let $L(n)$ be the polyhedron whose boundary is obtained by gluing $L_1^n$ to $L_2^n$ along their $S^1$ boundary in the only way which produces a cellular decomposition of the sphere with trivalent $1$--skeleton.  This family of combinatorial polyhedra evidently has no prismatic $3$ or $4$--circuits and so each has a geometric realization as a right-angled hyperbolic polyhedron.  For a geometric construction of these polyhedra in $\HH^3$, see \cite{V}.  These polyhedra $L(n)$ are called {\it L\"obell polyhedra}.  In particular, $L(5)$ is isomorphic to a dodecahedron.  


The first example of a closed orientable hyperbolic $3$--manifold was constructed by F L\"obell by gluing the faces of eight copies of $L(6)$.  This construction is an example of a more general procedure which produces eightfold manifold covers of right-angled hyperbolic polyhedra which will be described in detail below.

L\"obell polyhedra are fairly well understood.  In particular, their symmetry allows for an explicit computation of their volumes.  This was carried out by A Vesnin \cite{V} whose result will be recorded in the following:

\begin{theorem} For $n \geq 5$,
$$\vol(L(n)) = \frac{n}{2} \left(2 \Lambda(\theta_n) + \Lambda\left(\theta_n + \frac{\pi}{n}\right) + \Lambda\left(\theta_n - \frac{\pi}{n}\right) - \Lambda\left(2\theta_n - \frac{\pi}{2}\right)\right)$$
where $$\theta_n = \frac{\pi}{2} - \arccos\left(\frac{1}{2\cos(\frac{\pi}{n})}\right)$$ 
and $\Lambda: \mathbb{R} \rightarrow \mathbb{R}$ is the Lobachevskii function 
$$\Lambda(z) = - \int_0^z \log|2 \sin(t)| dt.$$
\end{theorem}

\begin{theorem} \label{lobellincreasing} $\vol(L(n))$ is an increasing function of $n$.
\end{theorem}
\begin{proof} Let $v(x)$ denote the function:
$$v(x) = \frac{x}{2} \left(2 \Lambda \left(\theta_x\right) + \Lambda\left(\theta_x + \frac{\pi}{x}\right) + \Lambda\left(\theta_x - \frac{\pi}{x}\right) - \Lambda\left(2\theta_x - \frac{\pi}{2}\right)\right)$$
where
$$\theta_x = \frac{\pi}{2} - \arccos\left(\frac{1}{2\cos\left(\frac{\pi}{x}\right)}\right).$$
The result will be shown by proving that $v$ is an increasing function for $x \geq 5$.

Here are a smattering of estimates which will be of use and which are stated without their elementary proofs:

\begin{lemma} For $x \geq 5$,

\begin{enumerate}
\item $\frac{\pi}{6} < \theta_x < \frac{\pi}{4}$.

\item $-\frac{\pi}{6} <  2\theta_x - \frac{\pi}{2} < 0 $.

\item $\frac{\pi}{6} < \theta_x + \frac{\pi}{x} < \frac{\pi}{2}$.

\item $\theta'_x < 0$.

\item $\Lambda$ is increasing on $(-\frac{\pi}{6}, \frac{\pi}{6})$ and decreasing on $(\frac{\pi}{6}, \frac{5\pi}{6})$.

\item $\Lambda'$ is decreasing on $(0, \frac{\pi}{2})$.

\end{enumerate}
\end{lemma}
\qed





Let 
$$g(x) = 2\Lambda(\theta_x) + \Lambda\left(\theta_x + \frac{\pi}{x}\right) + \Lambda\left(\theta_x - \frac{\pi}{x}\right) - \Lambda\left(2\theta_x - \frac{\pi}{2}\right)$$
so that $v(x) = \frac{x}{2}(g(x))$.  The function $v$ will be shown to be increasing for $x \geq 5$ by showing that $g$ is increasing on this interval.

Now calculate $g'(x)$:
\begin{align*}
g'(x) = & \quad   2\Lambda'(\theta_x)(\theta'_x) + \Lambda'\left(\theta_x + \frac{\pi}{x} \right) \left(\theta'_x - \frac{\pi}{x^2} \right)& \nonumber \\
& \quad \quad + \Lambda'\left(\theta_x - \frac{\pi}{x} \right) \left(\theta'_x + \frac{\pi}{x^2} \right) - \Lambda'\left(2\theta_x - \frac{\pi}{2}\right) (2\theta'_x)& \nonumber\\
\geq &\quad  \Lambda'\left(\theta_x + \frac{\pi}{x} \right) \left(\theta'_x - \frac{\pi}{x^2} \right) + \Lambda'\left(\theta_x - \frac{\pi}{x} \right) \left(\theta'_x + \frac{\pi}{x^2} \right) - \Lambda'\left(2\theta_x - \frac{\pi}{2}\right) (2\theta'_x)& \\
> & \quad \Lambda'\left(\theta_x - \frac{\pi}{x} \right) \left(\theta'_x - \frac{\pi}{x^2} \right) + \Lambda'\left(\theta_x - \frac{\pi}{x} \right) \left(\theta'_x + \frac{\pi}{x^2} \right) - \Lambda'\left(2\theta_x - \frac{\pi}{2}\right) (2\theta'_x)& \\
= & \quad 2\Lambda'\left(\theta_x - \frac{\pi}{x}\right)\theta'_x - \Lambda'\left(2\theta_x - \frac{\pi}{2}\right) (2\theta'_x)& \\
= & \quad 2\theta'_x\left(\Lambda'\left(\theta_x - \frac{\pi}{x} \right) - \Lambda'\left(2\theta_x - \frac{\pi}{2}\right)\right)& \\
 > & \quad 2\theta'_x\left(\Lambda'\left(\theta_x - \frac{\pi}{x} \right)\right)& \\
 > & \quad 0. &
\end{align*}

The first inequality follows from statements (1), (4) and (5) of the Lemma which imply $\Lambda'(\theta_x)\theta'_x$ is positive.  The second inequality follows from statement (4) and (6) of the Lemma.  The third inequality follows from statement (4) which implies that $2\theta'_x$ is negative, and statements (2) and (5) which imply $\Lambda'\left(2\theta_x - \frac{\pi}{2}\right)$ is positive.  The final inequality follows from (4), (1) and (5).

This computation implies that $g$, and therefore $v$, is an increasing function for $x \geq 5$.  Therefore, $\vol(L(n)) = v(n)$ is an increasing function for $n \geq 5$.  This ends the proof of Theorem \ref{lobellincreasing}.
\end{proof}

For reference, the first few values of $\vol(L(n))$ as computed by {\it Mathematica} are recorded in Table 1.

\begin{table}
\label{tab:lobellvol}
\begin{center}

\begin{tabular}{|r|r||r|r|}
\hline
$n$&$\vol(L(n))$&$n$&$\vol(L(n))$ \\
\hline
5&4.306...&13&15.822...\\
6&6.023...&14&17.140...\\
7&7.563...&15&18.452...\\
8&9.019...&16&19.758...\\
9&10.426...&17&21.059...\\
10&11.801...&18&22.356...\\
11&13.156...&19&23.651...\\
12&14.494...&20&24.943...\\
\hline
\end{tabular}

\end{center}
\caption{The volumes of the first sixteen L\"obell polyhedra.}
\end{table}

\medskip

\noindent {\bf Doubling:} Given a right-angled polyhedron $P$ and a face $F$ of $P$, a new right angled polyhedron called the {\it double of $P$ across $F$} or simply a {\it double} $dP$ can be constructed as follows.  Let $r_F$ be the reflection of $\HH^3$ across the plane supporting the face $F$.  Then $dP$ is defined to be $P \cup r_F(P)$.  Note that the edges of $F$ disappear in $dP$, in the sense that the dihedral angle along these geodesic segments in $dP$ has measure $\pi$.  

\medskip

\noindent {\bf Composition:} Let $P_1$ and $P_2$ be a pair of combinatorial polyhedra.  Suppose that $F_1$ is a face of $P_1$ that is combinatorially isomorphic to a face $F_2$ of $P_2$, which just means they have the same number of edges.  A new combinatorial polyhedron $P$ can be defined by choosing an isomorphism between $F_1$ and $F_2$, identifying $P_1$ and $P_2$ along $F_i$ using this isomorphism, deleting the interiors of the edges corresponding to $F_i$, and demoting the endpoints of these edges to nonvertices.  This new polyhedron is called a {\it composition} of $P_1$ and $P_2$.  Note that this composition has a distinguished prismatic $k$--circuit made up of the edges whose interiors were deleted where $k$ is the number of edges of $F_i$.  

\begin{theorem} The composition $P$ of a pair of right-angled hyperbolic polyhedra $P_1$ and $P_2$ is also a right-angled hyperbolic polyhedron.
\end{theorem}

\begin{proof}  It must be shown that $P$ has trivalent 1--skeleton and contains no prismatic 3 or 4--circuits.  That $P$ has trivalent 1--skeleton is clear.  Let $F_1 \subset P_1$ and $F_2 \subset P_2$ be the faces being identified in the formation of $P$ and let $c$ denote the distinguished prismatic $k$--circuit of $P$.  Suppose $P$ contains a prismatic $3$ or $4$--circuit $d$.  If $d$ can be made to completely miss $c$ by an isotopy that does not change the set of edges that $d$ intersects, then evidently one of $P_1$ or $P_2$ contains a prismatic $3$ or $4$--circuit, a contradiction.  

So then $d$ must intersect $c$, and it must do so in two distinct faces as $d$ and $c$ are simple closed curves on $S^2$.  This curve $d$ in the composition $P$ determines an arc in either $P_1$ or $P_2$ whose endpoints lie on $c$.  Furthermore, this arc intersects at most 3 or 4 edges of $P_i$, two of which are edges of $F_i$.  Closing this arc by adding a line segment in the face $F_i \subset P_i$ joining the endpoints of this arc, produces a prismatic 2--, 3-- or 4--circuit in $P_i$.  This is a contradiction.  
\end{proof}

Define {\it decomposition} to be the operation inverse to composition.  That is, decomposition splits a polyhedron $P$ along some prismatic $k$--circuit into a pair of polyhedra $P_1$ and $P_2$ each of which admit a right-angled hyperbolic structure.  A right-angled hyperbolic polyhedron which admits a decomposition is {\it decomposable}.  

A necessary condition for a polyhedron $P$ to be decomposable is that it have a prismatic $k$--circuit $c$ with $k \geq 5$ such that if $F$ is a face of $P$ which $c$ intersects, then, in $F$, the curve $c$ bounds combinatorial polygons on either side which have at least 5 edges.  In this case, $c$ will be said to have no {\it flats}.  However, this condition is not sufficient for decomposability as there is no guarantee that the resulting combinatorial polyhedra $P_1$ and $P_2$ obtained by splitting $P$ along $c$ admit a right-angled hyperbolic structure.  In particular, the polyhedra $P_1$ and $P_2$ may contain prismatic 3 or 4--circuits.  However, if $k = 5$, then the necessary condition is sufficient:

\begin{theorem}\label{p5cdec} Suppose $c$ is a prismatic 5--circuit of a right-angled hyperbolic polyhedron $P$ with no flats.  Then the polyhedron $P$ is decomposable along $c$.  
\end{theorem}

\begin{proof}  Denote the combinatorial polyhedra obtained by splitting $P$ along $c$ by $P_1$ and $P_2$.  It must be shown that these polyhedra admit right-angled hyperbolic structures.  It suffices to show this for just one of the two.  

It is obvious that $P_1$ has trivalent 1--skeleton.   

For the purposes of establishing a contradiction, suppose $d$ is a prismatic 3 or 4--circuit of $P_1$.  Let $F_1$ be the pentagonal face of $P_1$ produced by splitting $P$ along $c$.  If $d$ does not intersect $F_1$, then the curve $d$ persists in the composition $P$ which is a contradiction as it implies that $P$ has a prismatic 3 or 4--circuit.

So suppose $d$ intersects the pentagon $F_1$ in a pair of edges $d_1$ and $d_2$.  These edges must be disjoint in $F_1$, and since $F_1$ is a pentagon, there must be an edge $e$ of $F_1$ adjacent to both $d_1$ and $d_2$.  For $i=1,2$, let $e_i$ denote the edges of $P_1$ which share an endpoint with both $e$ and $d_i$.  These edges $e_i$ both belong to some face of $P_1$ called $E$.  

Via an isotopy, the circuit $d$ can be pushed across the edge $e$ to form a new circuit $\hat{d}$ so that instead of intersecting $F_1$, it intersects $E$, now in the edges $e_i$.  Let $D_i$ denote the face of $P_1$ which is adjacent to $E$ in the edge $e_i$.  Note that the faces $D_i$ and $E$ and the curve $\hat{d}$ are disjoint from the face $F_1$, and so persist in $P$.  By abusing of notation, label all edges and faces and curves in $P$ by their given labels in $P_1$.  

Suppose $d$ is a prismatic 3--circuit in $P_1$.  Then $\hat{d}$ is a 3--circuit in $P$ which intersects the three faces $D_1$, $E$ and $D_2$.  Therefore, these faces are pairwise adjacent.  Thus by Lemma \ref{threeadj}, these faces all share a vertex in $P$.  This implies $e_1$ and $e_2$ share a vertex, which implies that the edges $e$, $e_1$, and $e_2$ of $P_1$ form a triangle.  This contradicts the hypothesis on $c$.  

Suppose $d$ is a prismatic 4--circuit in $P_1$.  Then $\hat{d}$ is a 4--circuit in $P$ which intersects the three faces $D_1$, $E$ and $D_2$ as well as some fourth face $G$ which is adjacent to each $D_i$.  Therefore, by Lemma \ref{fouradj}, $G$ must be adjacent to $E$.  This implies that the edges $e$, $e_1$, $e_2$ and $G \cap E$ of $P_1$ form a quadrilateral.  This contradicts the hypothesis on $c$.

Therefore, $P_1$ contains no prismatic 3 or 4--circuits and has trivalent 1--skeleton which implies, by Andreev's theorem, that it admits a right-angled hyperbolic structure.
\end{proof}

Because there is a nontrivial moduli space of right-angled polygons \cite{CM}, this combinatorial composition of polyhedra is not as simple to understand in the geometric setting as the doubling process described above.  The difficulty is that $F_1$ and $F_2$ may be nonisometric in the geometric realizations of $P_1$ and $P_2$ and therefore the composition's geometric realization is not simply $P_1$ glued to $P_2$.  However, doubles are a special case of this welding operation where the combinatorics and the geometry agree.

\section{Manifold covers}

In this section, manifold covers of right-angled hyperbolic polyhedra, viewed as hyperbolic Coxeter orbifolds, are constructed.

If $P$ is a right-angled hyperbolic polyhedron, let the {\it reflection group of P}, denoted $\Gamma_P$, be the group generated by reflections in the planes supporting the faces of $P$.  Then evidently $\Gamma_P$ is a discrete group of isometries, and $\HH^3 / \Gamma_P = P$, giving $P$ the structure of a Coxeter orbifold.

A group presentation for $\Gamma_P$ is a simple matter to write down.  Let $F_1, \dots, F_k$ be the faces of $P$ listed in no particular order, and let $r_i$ denote the reflection in the plane supporting the face $F_i$.  Then 
$$\Gamma_P = < r_1, \dots, r_k \mid r_i^2 = 1, (r_i r_j)^2 = 1 \, \text{if $F_i$ is adjacent to $F_j$}>.$$
This presentation will be called the {\it standard presentation} of $\Gamma_P$.


The following theorem was originally proved by Mednykh and Vesnin \cite{MV}.  The proof contained herein is only slightly modified from their original argument.

\begin{theorem}\label{8foldcover} Every right-angled hyperbolic polyhedron $P = \HH^3 / \Gamma_P$ has an eightfold manifold cover.
\end{theorem}

\begin{proof}  Let $P$ be a right-angled hyperbolic polyhedron and $\Gamma_P$ be its reflection group.  Let $g: \Gamma_P \rightarrow \mathbb{Z}/2\mathbb{Z}$ be the group homomorphism which gives the mod 2 length of a word in $\Gamma_P$.  Here the group presentation for $\Gamma_P$ given above is used.  Let $G_P$ be the kernel of this homomorphism.

This group $G_P$ determines a double cover of $P$ which is easy to visualize.  Take two copies of $P$ and identify each face of one copy to the corresponding face of the other copy.  The resulting geometric object is a hyperbolic orbifold which is topologically $S^3$ with a one dimensional singular set which is isomorphic as a graph to the 1--skeleton of $P$.  As $P$ has right angles, the cone angle around each edge of this singular set is $\pi$.  

A Wirtinger type presentation then gives a presentation for $G_P$.  A loop around an edge in the singular set corresponding to an edge in $P$ which is contained by the faces $F_i$ and $F_j$ gives the generator of $G_P$ given by $a_{ij} = r_i r_j$.  By the relations of $\Gamma_P$, $a_{ij} = a_{ji}$.  Note that $a_{ij}$ is a composition of reflections in orthogonal planes and so is a rotation of $\pi$ about the geodesic which supports the edge $F_i \cap F_j$.  In particular, there are relations $a_{ij}^2 = 1$.  Further relations are given by vertices so that if $F_i$, $F_j$ and $F_k$ are distinct faces sharing a vertex, then $a_{ij} a_{jk} = a_{ik}$.

To find a torsion free subgroup of $G_P$ that has index four, some facts about colorings of the faces and edges of $P$ will need to be collected.  Suppose the faces of $P$ are colored by the four elements of the group $(\mathbb{Z}/ 2 \mathbb{Z})^2$ with the usual condition that if two faces are adjacent, they are colored differently.  Such a face coloring is guaranteed by the four color theorem proved by Appel and Haken \cite{AH}.  Then each edge of $P$ can be colored by the sum of the colors of the faces which contain the edge.  Note that each edge is then colored by one of the three nonzero elements of $(\mathbb{Z} / 2 \mathbb{Z})^2$.  Note also that the sum of two distinct nonzero elements of $(\mathbb{Z} / 2 \mathbb{Z})^2$ is the third nonzero element.

\begin{lemma} This is an edge 3--coloring of the 1--skeleton of $P$.  That is, no two edges which are colored the same share an endpoint.
\end{lemma}
\begin{proof} Suppose that two edges $e_1$ and $e_2$ share a vertex $v$.   By trivalence, there is a unique edge $d$ which also has vertex $v$ which is not $e_1$ or $e_2$.  Let $A$ be the face containing $e_1$ and $e_2$, $B$ the face containing $e_2$ and $d$, and $C$ the face containing $d$ and $e_1$.  Let $L_F$ denote the color of the face $F$ given by the face coloring.  Then the color of the edge $e_1$ is $L_A + L_C$ while the color of $e_2$ is $L_A + L_B$.  Since $L_B \neq L_C$, these colors are different.  
\end{proof}

Let $C$ denote this coloring of $P$.  Define a homomorphism $h: G_P \rightarrow (\mathbb{Z} / 2 \mathbb{Z})^2$ which sends the generator $a_{ij}$ to the coloring of the edge shared by the faces $F_i$ and $F_j$.  That this assignment of images for generators of $G_P$ extends to a group homomorphism follows from the comments above.  Let $H_{(P,C)}$ denote the kernel of $h$.

\begin{lemma}\label{Htorsionfree} $H_{(P,C)}$ is torsion free.
\end{lemma}
\begin{proof} The proof of this result is an induction argument on the length of a freely reduced word in the generators $a_{ij}$.  If $w$ is such a word, let $l(w)$ denote its length.  Let $e_{ij}$ denote the edge of $P$ corresponding to the generator $a_{ij}$.  The proof will show that if $w$ is in $H_{(P,C)}$ and is not the identity element, then $w$ is a hyperbolic screw translation, or a {\it loxodromic} isometry in the vernacular.  These isometries have infinite order, and also have the property that the composition of two such is again loxodromic.  

First note that if $l(w) = 1$, then $w$ cannot be in $H_{(P,C)}$.  Suppose then that $l(w) = 2$, say $w = a_{ij} a_{mn}$.  If $w$ is in $H_{(P,C)}$, then $h(a_{ij}) = h(a_{mn})$.  Therefore, as $w$ is freely reduced, $e_{ij}$ and $e_{mn}$ do not share any vertices.  
Let $g$ denote the geodesic which intersects $e_{ij}$ and $e_{mn}$ orthogonally.  Then the isometry $w = a_{ij} a_{mn}$ is loxodromic and its translation axis is $g$.  Therefore, in particular, $w$ cannot have finite order.

Suppose $l(w) = 3$ with $h(w) = 0$, say $w = abc$.  Note that $h(a)$, $h(b)$ and $h(c)$ must all be different nonzero elements of $(\mathbb{Z} / 2 \mathbb{Z})^2$.  Furthermore, note that if the edges associated to the generators $a$ and $b$ share a vertex, then $w$ can be reduced in the group to an equivalent word of length 2, and thus is loxodromic.  So assume these edges are disjoint.  Then, $ab$ is a loxodromic isometry and so $w$ is the composition of a loxodromic isometry with the rotation $c$ and is therefore loxodromic. 

Now assume, for the purposes of induction, that if $l(w) \leq d$ for $d \geq 3$ with $w \in H_P$, then $w$ is loxodromic and so does not have finite order.  Let $w$ be a freely reduced word of length $d+1$ with $h(w) = 0$.  Let $x$ be the prefix of $w$ of length $\lfloor \frac{l(w)}{2} \rfloor$ where $\lfloor \cdotp \rfloor$ denotes the floor function.  Furthermore, let $y$ be the freely reduced word representing $x^{-1} w$ so that $w = xy$.  Note that since $l(w) \geq 4$, the lengths of both $x$ and $y$ are at least 2.

Suppose first that $h(x) = 0$.  Then evidently, $h(y) = 0$ and so $x$ and $y$ are words whose lengths are shorter than that of $w$ and live in $H_{(P,C)}$.  Therefore, by the induction hypothesis, $x$ and $y$ are loxodromic, and so $w$, their composition, is also loxodromic.  

Suppose then that $h(x) \neq 0$.  Let $a_{ij}$ be the first letter of the word $y$ and $z$ the freely reduced word representing $a_{ij}^{-1} y$.  If $h(x) = h(a_{ij})$, then $h(xa_{ij}) = 0$ and so $h(z) = 0$.  Therefore, induction says that $x a_{ij}$ and $z$ are loxodromic as they are strictly shorter than $w$, and so $w = x a_{ij} z = x y$, their composition, is also loxodromic.

Finally, suppose $h(x) \neq 0$ but $h(x) \neq h(a_{ij})$.  Let $e_{ik}$ and $e_{kj}$ be edges of $P$ which both share the same endpoint with $e_{ij}$ and let $a_{ik}$ and $a_{kj}$ denote the corresponding generators of $G_P$.  Note that one of $a_{ik}$ or $a_{kj}$ must be colored by the color $h(x)$.  Suppose without loss of generality that $h(a_{ik}) = h(x)$.   Then, in the word $w$, replace the first letter $a_{ij}$ in $y$ by $a_{ik} a_{kj}$.  Note that since $d+1 = l(w) \geq 4$, the words $x a_{ik}$ and $a_{kj} z$ both have length smaller than $d+1$ and at least 2. Then, by construction, $h(x a_{ik}) = 0$ and so $h(a_{kj} z) = 0$ and therefore both these words are loxodromic by the induction hypothesis.  Thus their composition $x a_{ik} a_{kj} z = x a_{ij} z = x y = w$ is loxodromic.  Note that the words $x a_{ik}$ and $a_{kj} z$ may not be freely reduced.  However, since free reduction only reduces length, the induction hypothesis still applies. 

This ends the proof of Lemma \ref{Htorsionfree}.
\end{proof}

Therefore, $H_{(P,C)}$ is a subgroup of index 4 in $G_P$ which has index 2 in $\Gamma_P$, and so $H_{(P,C)}$ is a torsion free subgroup of index 8 in $\Gamma_P$.  This proves Theorem \ref{8foldcover}. 
\end{proof}

Suppose $F$ is some face of $P$.  Let $\Gamma_{(P,F)}$ be the group generated by reflections in the planes supporting each face of $P$ except $F$.  This is a subgroup of index 2 in $\Gamma_P$.  Let $\HH_F$ denote the closed connected subset of $\HH^3$ which contains $P$ and is bounded by the planes supporting $F$ and its images under the action of $\Gamma_{(P,F)}$.  Then the group $\Gamma_{(P,F)}$ acts on $\HH_F$ with quotient $\HH_F / \Gamma_{(P,F)} = P$.  The orbifold structure on $P$ can be thought of as mirroring every face of $P$ except $F$. The face $F$ is the totally geodesic boundary of this orbifold structure on $P$.

Restricting the homomorphism $g$ which computes the mod 2 length of a word gives a map $\Gamma_{(P,F)} \rightarrow \mathbb{Z}/ 2 \mathbb{Z}$.  The kernel of this map will be denoted $G_{(P,F)}$.  This group acts on $\HH_F$ and the quotient orbifold can be visualized in the following way.  Take two copies of $P$, and identify each face of one copy of $P$ to the corresponding face in the other copy, but do not perform this identification if the face is $F$.  Then the resulting space is topologically a 3--ball $B^3$.  The orbifold singularities are a graph isomorphic to the 1--skeleton of $P$ with the edges contained in the face $F$ removed.  The resulting graph has some number of 1--valent vertices which live on the boundary of the 3--ball.  In fact, in the hyperbolic metric, the boundary of the 3--ball is a totally geodesic hyperbolic 2--orbifold.  The edges in the singularity graph all have cone angle $\pi$ as do the cone points on the boundary corresponding to the 1--valent vertices.

As above, a Wirtinger type presentation gives a presentation for $G_{(P,F)}$ where each edge of $P$ but not in $F$ gives a generator of order 2, and additional relations given by the vertices.  In fact, it is evident from the presentations that $G_{(P,F)} = G_P \cap \Gamma_{(P,F)}.$  

Let the faces of $P$ other than $F$ be colored by elements of $(\mathbb{Z} / 2 \mathbb{Z})^2$, and the edges of $P$ other than those of $F$ be colored by the sum of the colors which contain the edge as above.  Call this coloring $C$.  Let $h$ be the homomorphism $h: G_{(P,F)} \rightarrow (\mathbb{Z} / 2 \mathbb{Z})^2$ sending a generator to the color of its corresponding edge.  Let $H_{(P,C,F)}$ be the kernel of $h$, a subgroup of index 8 of $\Gamma_{(P,F)}$.

\begin{theorem}\label{Htgtorsionfree}  $H_{(P,C,F)}$ is torsion free.
\end{theorem}

\begin{proof} The proof is essentially the same as that of Theorem \ref{Htorsionfree} and so will be suppressed.  
\end{proof}

Let $M_{(P,C,F)}$ denote the hyperbolic manifold with geodesic boundary $\HH_F / H_{(P,C,F)}$ for some choice of coloring $C$ of $P \setminus F$.  This is an eightfold cover of the hyperbolic orbifold with geodesic boundary $\HH_F / \Gamma_{(P,F)}$.  



\section{Decomposition}

Let $P_1$ and $P_2$ be a pair of right-angled hyperbolic polyhedra and $F_1 \subset P_1$, $F_2 \subset P_2$ a pair of faces which are isomorphic $k$--gons for some choice of isomorphism.  Let $P$ denote the composition of $P_1$ and $P_2$ along these faces using this isomorphism as described in Section 4.  The polyhedron $P$ has a distinguished prismatic $k$--circuit $c$ which, as a simple closed curve, consists of the edges of $F_1$ and $F_2$ that were identified and whose interiors were deleted in the formation of the composition $P$.  Let $C_1, \dots, C_k$ denote the faces of $P$ which this prismatic $k$--circuit $c$ intersects.  
This curve $c$ bounds in $P$ an embedded topological $k$--gon which will be denoted $F$.

For $i=1,2$, let $\phi_i : P_i \rightarrow P$ denote the topological embedding induced by the composition in the obvious way.  This map sends faces of $P_i$ which are not adjacent to $F_i$ to faces of $P$ by cellular isomorphisms, sends faces of $P_i$ adjacent to $F_i$ into $C_j$ for some $j$, and sends the face $F_i$ to the suborbifold $F$.


In $\Gamma_{(P_i, F_i)}$, the subgroup generated by reflections in the faces adjacent to $F_i$ (or indeed any face of $P_i$) is itself a reflection group of a right-angled polygon.   Let this group be denoted $\Gamma_{F_i}$.  Then, the isomorphism which identifies $F_1$ with $F_2$ in the formation of the composition $P$ induces a group isomorphism between $\Gamma_{F_1}$ and $\Gamma_{F_2}$.  Let $G$ denote the isomorphism class of these groups.

\begin{theorem}\label{compositionamalgamation} If $P$ is the composition of $P_1$ and $P_2$ along the faces $F_1 \subset P_1$ and $F_2 \subset P_2$, then $\Gamma_P$ is isomorphic to the free product with amalgamation $$ \Gamma_P \cong \Gamma_{(P_1,F_1)} \ast_G \Gamma_{(P_2, F_2)}.$$
\end{theorem}

\begin{proof}  Suppose that $F_1$ and $F_2$ are polygons with $k$ edges.  Let $\{s_j\}$ and $\{t_j\}$ denote the generators of $\Gamma_{(P_1,F_1)}$ and $\Gamma_{(P_2, F_2)}$ respectively, which correspond to reflections in the planes supporting each face of $P_i$ except $F_i$.  Index these generators in such a way that $\{s_1, \dots, s_k\}$ and $\{t_1, \dots t_k\}$ correspond to the faces of $P_1$ and $P_2$ (resp.) which are adjacent to $F_1$ and $F_2$ (resp.).  Furthermore, index in such a way that for each $j = 1, \dots k$, $s_j$ and $t_j$ correspond to faces adjacent to $F_1$ and $F_2$ (resp.) which are mapped into the face $C_j$ in $P$ under the embeddings $\phi_i$ of $P_i$ into $P$.   

Then a presentation for the free product with amalgamation $\Gamma_{(P_1,F_1)} \ast_G \Gamma_{(P_2, F_2)}$ is given by a generating set $\mathcal{S} = \{s_j\} \cup \{t_j\}$ with relations $\mathcal{R}$ given by three types of words:

\begin{enumerate}
\item $(s_j)^2$ and $(t_j)^2$ for all $j$.

\item $(s_i s_j)^2$ and $(t_i t_j)^2$ if the corresponding faces in $P_1$ and $P_2$ (resp.) are adjacent.

\item $s_j = t_j$ for $j = 1, \dots k$.
\end{enumerate}

Define a homomorphism $\psi$ from $\Gamma_{(P_1,F_1)} \ast_G \Gamma_{(P_2, F_2)}$ to $\Gamma_P$ in the following way.  If $u \in \mathcal{S}$ corresponds to a face $U \subset P_i$, then send $u$ to the generator of $\Gamma_P$ given by the reflection in the face of $P$ containing the image of $U$ under $\phi_i$.  That this assignment of generators extends to a homomorphism is obvious using the standard presentation for $\Gamma_P$.  

To show that $\psi$ is an isomorphism, an inverse map will be defined.  Let $\{r_i\}$ denote the generators of $\Gamma_P$ in the standard presentation with $r_1, \dots, r_k$ denoting the reflections in the faces $C_1, \dots, C_k$ respectively.  Send a generator $r_j$ with $j > k$ to the generator of the amalgam which was sent to $r_j$ under $\psi$, and send $r_j$ with $1 < j < k$, to $s_j = t_j$.  This map of generators extends to a homomorphism which is inverse to $\psi$.
\end{proof}

For $i = 1,2$, let $j_i : \Gamma_{(P_i, F_i)} \hookrightarrow \Gamma_P$ denote the natural inclusions.  These homomorphisms are induced by the embeddings $\phi_i$.  Similarly, the natural inclusion $k : G \hookrightarrow \Gamma_P$ is induced by the embedding of $F$ into $P$.  Note that $F$ is isomorphic in the category of orbifolds to the quotient of the plane supporting $F_i$ by the action of $\Gamma_{F_i} \cong G$ for either $i=1$ or $i=2$.  In particular, because $k$ is injective, it is natural to think of this embedded suborbifold as being incompressible in $P$.

Consider the manifold $M_{(P,C)} = \HH^3 / H_{(P,C)}$ which is an eightfold cover of the orbifold $P = \HH^3 / \Gamma_P$.  Denote the covering projection by $\rho$.  Let $\Sigma$ denote the closed surface embedded in $M_{(P,C)}$ which covers the embedded suborbifold $F \subset P$.  Then the fundamental group of each component of $\Sigma$ is isomorphic to $G \cap H_{(P,C)}$ and $\Sigma$ is incompressible in $M_{(P,C)}$.  That is, $M_{(P,C)}$ is Haken.

Consider the manifold $M_{(P,C)} - \mathcal{N}(\Sigma)$ obtained by splitting $M_{(P,C)}$ along $\Sigma$ by removing a tubular neighborhood $\mathcal{N}(\Sigma)$.  Since $\Sigma$ covers the separating suborbifold $F \subset P$, the surface $\Sigma$ is separating in $M_{(P,C)}$.  Therefore, $M_{(P,C)} - \mathcal{N}(\Sigma)$ is a pair of manifolds $M_1$ and $M_2$ with boundary homeomorphic to $\Sigma$.  Furthermore, since $F$ splits $P$ into two parts given by $\phi_i(P_i)$ for $i = 1,2$, restricting the covering map $\rho$ to $M_i$ is itself a covering map over $\phi_i(P_i)$ and, in particular, is an eightfold covering map.  It is the eightfold covering map corresponding to the group $H_{(P_i, C_i, F_i)}$ where $C_i$ is the coloring of $P_i \setminus F_i$ gotten by restricting the coloring $C$ of $P$ to the image of the embedding $\phi_i$.

Therefore, by rigidity, $M_i$ is homeomorphic to $M_{(P_i, C_i, F_i)} = \HH_{F_i} / H_{(P_i, C_i, F_i)}$.  In English, this says that $M_i$ is homeomorphic to an eightfold orbifold cover of the polyhedron $P_i$.  All of this will be recorded in the following theorem:

\begin{theorem} \label{Mhaken} The manifold $M_{(P,C)}$ is a Haken manifold with a separating incompressible surface $\Sigma$.  Splitting $M_{(P,C)}$ along $\Sigma$ produces a pair of manifolds $M_i$, $i=1,2$ with $$M_i \cong M_{(P_i, C_i, F_i)} = \HH_{F_i} / H_{(P_i, C_i, F_i)}.$$
\end{theorem}
\qed

The following result, proved by Agol, Storm and Thurston \cite{ADST}, gives a description of the effect of decomposition on volumes.  

\begin{theorem}\label{AST} If $M$ is a closed hyperbolic Haken 3--manifold and $\Sigma \subset M$ is an incompressible surface, then 
$$\vol(M) \geq \frac{1}{2} V_3 ||D(M - \mathcal{N}(\Sigma))||$$
where $V_3$ denotes the volume of the regular ideal tetrahedron, $||\cdotp||$ denotes the Gromov norm, and $D()$ denotes manifold doubling.
\end{theorem}

If $M - \mathcal{N}(\Sigma)$ admits a hyperbolic structure with totally geodesic boundary, then $D(M - \Sigma)$ is a disjoint union of closed hyperbolic manifolds and so the simplicial volume $V_3 ||D(M - \mathcal{N}(\Sigma))||$ coincides with the hyperbolic volume $\vol(D(M - \mathcal{N}(\Sigma)) = 2 \vol(M - \mathcal{N}(\Sigma))$.  Therefore the above theorem implies that $\vol(M) \geq \vol(M - \mathcal{N}(\Sigma))$.

In the particular case of the manifold $M_{(P,C)}$ which when split along its incompressible surface $\Sigma$ produces the manifolds $M_1$ and $M_2$, Theorem \ref{AST} implies that $\vol(M_{(P,C)}) \geq \vol(M_1) + \vol(M_2)$.  Since $M_{(P,C)}$, $M_1$ and $M_2$ are eightfold covers of the polyhedra $P$, $P_1$ and $P_2$ respectively, the following result follows immediately:

\begin{theorem} \label{decompvoldecresing} If $P$ is a right-angled hyperbolic polyhedron which is the composition of right-angled hyperbolic polyhedra $P_1$ and $P_2$, then 
$$\vol(P) \geq \vol(P_1) + \vol(P_2).$$ \qed

\end{theorem}

\section{Edge Surgery}

This section will be devoted to defining and studying a simple combinatorial operation on right-angled polyhedra called {\it edge surgery}.  This operation together with decomposition will simplify any given right-angled hyperbolic polyhedron into a set of L\"obell polyhedra.  First, some definitions:

Two distinct faces $F_1$ and $F_2$ of $P$ are said to be {\it
edge connected} if they are nonadjacent and there exists an edge $e$ of $P$
connecting a vertex of $F_1$ to a vertex of $F_2$.  Such an edge $e$ is said
to {\it edge connect} $F_1$ and $F_2$.  Note that since $P$ has trivalent 1--skeleton, every edge of $P$ edge connects a unique pair of faces.

A face $F$ of $P$ is called {\it large} if it is has $6$ or
more edges.  An edge $e$ is called {\it good} if it edge connects two large
faces.  A good edge is called {\it very good} if it is not a part of any
prismatic $5$--circuit.

If $e$ is an edge of $P$, call the combinatorial operation of deleting the interior $e$ and demoting its endpoints to nonvertices {\it edge surgery along e} (see Figure \ref{fig: es}).  Call the line segment in $P_1$ that was removed the {\it trace} of the edge $e$ and the vertices that were demoted the {\it traces} of the vertices. 

\begin{figure}
	\centering
	\includegraphics[bb=0 0 210 97]{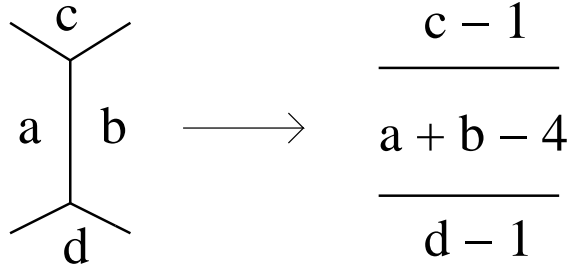}
	\caption{Edge surgery.  Here, the faces labelled c and d are edge connected.  The labels a,b,c,d also represent the number of edges of the corresponding face, and this figure shows the effect of edge surgery on combinatorics.}
	\label{fig: es}
\end{figure}

\begin{theorem} \label{edgesurgeryrightangled} If a right-angled hyperbolic polyhedron $P_0$ has a very good edge $e$ and $P_1$ is the result of edge surgery along $e$, then $P_1$ is also a right-angled hyperbolic polyhedron.
\end{theorem}

\begin{proof} The conditions outlined by Andreev's characterization of right-angled hyperbolic polyhedra must be checked.  That is, it must be shown that $P_1$ has trivalent $1$--skeleton and no prismatic $3$ or $4$--circuits.  

It is clear that $P_1$ has trivalent $1$--skeleton as the edge that
is deleted has its endpoints demoted to nonvertices while all other vertices of $P_0$ are left unaffected.

Suppose for the purposes of establishing a contradiction that $P_1$ contains a prismatic $3$ or $4$--circuit $c$.  If necessary, perturb $c$ slightly so that it does not intersect the traces of the endpoints of $e$.  Let $F$ denote the face of $P_1$ containing the trace of $e$.  If $c$ does not intersect the face $F$ at all, then it is clear that $P_0$ contains a prismatic $3$ or $4$--circuit which is a contradiction.  So suppose $c$ intersects $F$.  Then $c$ intersects the boundary of the polygon $F$ in exactly two edges $d_1$ and $d_2$.  

If $d_1$ or $d_2$ is an edge which contains the trace of an endpoint of $e$, then via an isotopy, $c$ can be made to intersect $d_1$ and $d_2$ on the same side of the trace of $e$ and then further pushed by an isotopy to miss the trace of $e$ completely.  Therefore the curve $c$ also determines a prismatic $3$ or $4$--circuit in $P_0$ which is a contradiction as $P_0$ is a right angled hyperbolic polyhedron.  A similar argument shows that $d_1$ and $d_2$ cannot be edges which lie on the same side of the trace of $e$.

Suppose then that $d_1$ and $d_2$ do not contain the trace of a vertex of $e$ and are on opposite sides of the trace of $e$.  Via an isotopy, $c$ can be made to intersect the trace of $e$ in exactly 1 point.  Then $c$, being a prismatic $3$ or $4$--circuit in $P_1$, determines a prismatic $4$ or $5$--circuit (resp.) in $P_0$ of which $e$ is a member.  This is a contradiction as $P_0$ cannot have a prismatic $4$--circuit and $e$ is very good.
\end{proof}

\begin{theorem}\label{goodexist} Let $P$ be a right-angled hyperbolic polyhedron which is not of L\"obell type.  Then either it has a good edge or it is decomposable.  If $P$ does not have a good edge, then $P$ is decomposable into a pair of right-angled polyhedra, one of which is a dodecahedron.
\end{theorem}

\begin{proof} Let $X$ be a maximal connected subset of $\partial P$ which contains only large faces.  Topologically, $X$ is a subsurface of $\partial P \cong S^2$ and therefore is homeomorphic to a sphere with $k$ disks removed.  By Lemma \ref{twelvepentagons}, $P$ must have at least 12 pentagons and if $P$ is not a dodecahedron, then it must have some number of faces which are not pentagons.  Thus $X \neq \emptyset$ and $k \geq 1$.  Let $D_1, D_2, \dots, D_k$ denote the disks of $\partial P \setminus X$ and label their boundaries $\partial D_i = S_i$.  

Note that the $S_i$ are combinatorially polygons, and by maximality in the way $X$ was defined, every edge of $S_i$ is an edge of a pentagon lying in the disk $D_i$.  Let $S$ and $D$ denote some fixed boundary/disk pair.  Suppose that $Q$ is a pentagon with at least one edge belonging to $S$.  The proof breaks down into many cases.
\smallskip

{\bf Case A:} Suppose first all $5$ edges of $Q$ lie on $S$.  Then evidently $D = Q$.  As every face in $X$ is large, any edge of $Q$ is a good edge.
\smallskip

{\bf Case B:} Suppose $4$ edges of $Q$ lie on $S$.  If the edges of $Q$ are labelled cyclicly by $e_i, i = 1,2, \ldots 5$, then without loss of generality assume $e_1, \ldots, e_4$ lie on $S$.  Then it is evident that $e_2$ and $e_3$ are both good edges.
\smallskip

{\bf Case C:} Suppose $3$ edges of $Q$ lie on $S$.  Again, label the edges of $Q$ cyclicly by $e_i$.  Suppose that these three edges are adjacent, without loss of generality, 
say, $e_1, e_2, e_3$ lie on $S$.  Then it is evident that $e_2$ is a good edge.  But there is another possibility.  It could be that, say, $e_1$, $e_2$, and 
$e_4$ lie on $S$.  In this case, both $e_3$ and $e_5$ are good edges.
\smallskip

{\bf Case D:} Suppose then that every pentagon with an edge lying on $S$ meets $S$ in one or two edges.  If $Q$ is a pentagon with exactly two nonadjacent edges lying on $S$, then the edge of $Q$ which is adjacent to both of them must be good.  Thus, suppose that every pentagon with exactly two edges lying on $S$ does so in a pair of adjacent edges. Call a pentagon {\it inward} if it intersects $S$ in one edge and {\it outward} if in two adjacent edges.  This case breaks down into a number of subcases:
\smallskip

{\bf Case D.1:} Suppose that all pentagons lying on $S$ are inward.  Then $X$ is a single large face of $P$ as every edge emanating from a vertex of $X$ passes into the interior of $D$.  Then $X$ along with the faces adjacent to $X$ form a pentagonal flower.

If $X$ is edge connected to some other large face in the interior of $D$, then a good edge exists.  So suppose then $X$ is edge connected only to pentagons.  These pentagons can be arranged along the boundary of the pentagonal flower in only one way so that $P$ must be the L\"obell polyhedron $L(n)$ where $n$ is the number of edges of $X$.  This is a contradiction.
\smallskip

{\bf Case D.2:} Suppose that some pentagons lying on $S$ are inward and some are outward.  The proof breaks down into even more subcases depending on the number of consecutive outward pentagons incident to $S$: 
\smallskip

{\bf Case D.2.1:} Suppose there is a set of three consecutive outward pentagons  $A$, $B$, and $C$ with $B$ adjacent to both $A$ and $C$.  Then there is a face $G$ adjacent to these pentagons lying in $D$.  If $G$ is large, then any edge which is the intersection of two of these pentagons is good.  

So suppose instead that $G$ is a pentagon.  Then $G$ is adjacent to $A$, $B$, $C$, and two other faces $H_1$ and $H_2$ which are adjacent to $A$ and $C$ respectively.  Each $H_i$ shares a vertex with $A$ or $C$ lying on $S$ and, therefore, must themselves lie on $S$ and so must be pentagons.  Note also that $H_1$ and $H_2$ are adjacent as they are each adjacent to $G$ in edges which are adjacent.

Suppose $H_1$ is an inward pentagon.  Label the edge of $H_1$ lying on $S$ by $e_1$.  Label the edge of $H_1$ adjacent to $e_1$ but which is not incident to $A$ by $f_1$.  Then there is a pentagon $J$ which is adjacent to $H_1$ in $f_1$ ($J$ cannot be adjacent to $A$ as $A$ is outward).  Note that the edge $g_1$ of $J$ which shares exactly one endpoint with both $e_1$ and $f_1$ must lie on $S$.  Note also that $J$ is adjacent to $H_2$ in some edge $f_2$ since $f_1$ contains a vertex of $H_2$.  So $H_2$ is inward and if $e_2$ denotes the edge of $H_2$ lying on $S$, then $e_2$ and $f_2$ are adjacent.  Note that the edge $g_2$ of $J$ which shares exactly one endpoint with both $e_2$ and $f_2$ must lie on $S$.  Thus $g_1$ and $g_2$ are nonadjacent edges in $J$ which lie on $S$ and so $J$ is neither inward nor outward.  This is a contradiction which shows that $H_1$ is outward.  A similar argument applies for $H_2$.


So $H_1$ and $H_2$ must both be outward.  Thus every pentagon lying on $S$ is outward, which is a contradiction.

\smallskip

{\bf Case D.2.2:} Suppose there is a set of exactly two consecutive outward pentagons, call them $A_1$ and $A_2$.  Let $B_1$ and $B_2$ be the inward pentagons lying on $S$ adjacent to $A_1$ and $A_2$ respectively. Let $C_1$ and $C_2$ be the pentagons lying on $S$ adjacent to $B_1$ and $B_2$ respectively.

There is a face $G$ which is adjacent to all four of $A_1$, $A_2$, $B_1$, $B_2$.  If $G$ is large, then the edge $A_1 \cap A_2$ is a good edge (as are $B_1 \cap A_1$ and $A_2 \cap B_2$).

So suppose $G$ is a pentagon.  Then there is an edge of $G$ which edge connects $B_1$ to $B_2$.  Let $H$ be the face adjacent to $G$ through this edge.  Then note that $H$ is adjacent to $B_1$ and $B_2$ and in particular, is edge connected to a large face in $X$ by the edges $C_1 \cap B_1$ and $B_2 \cap C_2$.  Therefore, if $H$ is large, then these edges are good.

So suppose $H$ is a pentagon.  Then $H$ is adjacent to $G$, $B_1$, $B_2$, $C_1$ and $C_2$.  In particular, $C_1$ and $C_2$ are adjacent pentagons.  An argument almost exactly like the one given in Case D.2.1, which showed that a pair of pentagons were outward, shows, in this case, that $C_1$ and $C_2$ are outward pentagons.



\begin{figure}
	\centering
	\includegraphics[bb=0 0 261 173]{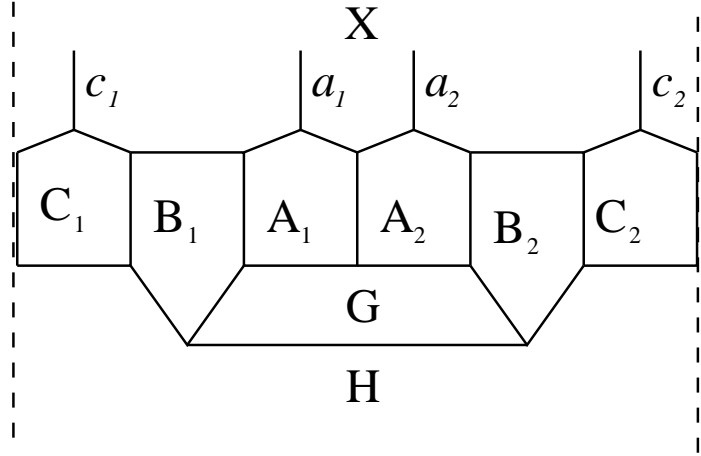}
	\caption{Case D.2.2 of Theorem \ref{goodexist}.  The dashed lines in this figure are identified.}
	\label{fig: cased22}
\end{figure}

Therefore, there are six pentagons lying on $S$, four of which are outward.  Consider the four edges in $X$ which emanate from vertices of these outward pentagons.  The claim is that these edges form a prismatic $4$--circuit.  Label the edges emanating from $A_i$ by $a_i$ and the edges emanating from $C_i$ by $c_i$.  See Figure \ref{fig: cased22}.

Note that $a_1$ and $a_2$ cannot share an endpoint as this would imply that the face in $X$ adjacent to $A_1$ and $A_2$ is a quadrilateral.  Similarly, $c_1$ and $c_2$ do not share an endpoint.  

If $a_i$ and $c_i$ share an endpoint, then the face in $X$ of which they are both edges is a pentagon which is a contradiction since $X$ was assumed to consist only of large faces.

Suppose $a_1$ and $c_2$ share an endpoint $v$.  By trivalence, there is another edge $d$ emanating from $v$.  Note that this edge $d$ cannot be $a_2$ as $a_1$ and $a_2$ cannot share endpoints by the above argument.  Call the face in $X$ for which $a_1$ and $a_2$ are edges $M$ and the face for which $a_2$ and $c_2$ are edges $N$.  Then $M$ and $N$ are adjacent via the edge $a_2$, but also intersect in at least $v$ which is disjoint from $a_2$.  This is a contradiction.  A similar argument shows that $a_2$ and $c_1$ cannot share an endpoint.

Therefore, $a_1$, $a_2$, $c_1$, and $c_2$ form a prismatic $4$--circuit which is a contradiction.
\smallskip

{\bf Case D.2.3:} Suppose that every outward pentagon lying on $S$ is adjacent to a pair of inward pentagons lying on $S$, one on each side.  Let $A$ be such an outward pentagon, and let $B_1$ and $B_2$ be the inward pentagons adjacent to $A$.  Let $C_1$ and $C_2$ be the pentagons lying on $S$ adjacent to $B_1$ and $B_2$ respectively (but are not $A$).  

There is an edge of $A$ which edge connects $B_1$ and $B_2$.  Call the face adjacent to $A$ in this edge $G$.  Note that the edges $A \cap B_1$ and $A \cap B_2$ edge connect $G$ to a large face in $X$.  Therefore, if $G$ is large then these edges are good.  

So suppose $G$ is a pentagon.  Then $G$ is adjacent to $A$, $B_1$ and $B_2$.  Let $H_1$ and $H_2$ denote the remaining two faces adjacent to $G$ with $H_1$ adjacent to $B_1$ and $H_2$ adjacent to $B_2$.  Note that for $i = 1,2$ the edge $B_i \cap C_i$ edge connects $H_i$ to some large face in $X$.  Thus if either $H_1$ or $H_2$ is large, then $P$ has a good edge.  

So suppose both $H_1$ and $H_2$ are pentagons.  Let $D$ denote the face which is adjacent to each of $C_1$, $C_2$, $H_1$ and $H_2$.  See Figure \ref{fig: d23}.  The remainder of this case breaks down into further subcases depending on the nature of the pentagons $C_1$ and $C_2$.  
\smallskip

\begin{figure}
	\centering
	\includegraphics[bb=0 0 287 209]{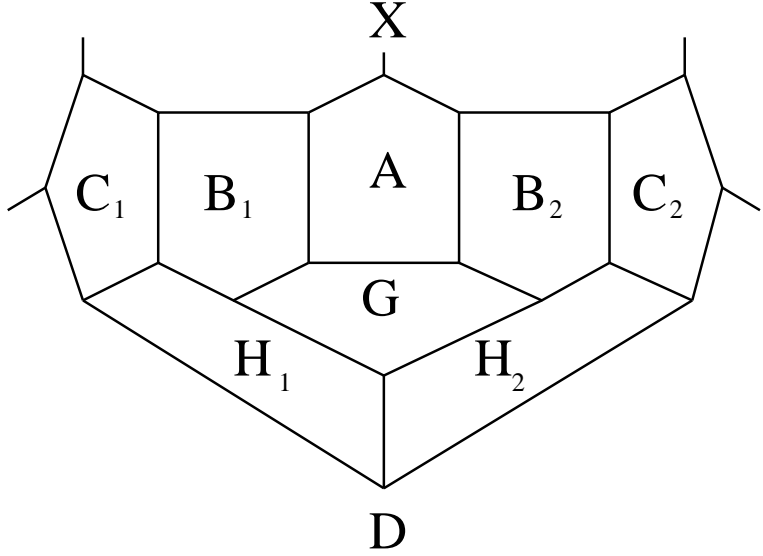}
	\caption{Case D.2.3 of Theorem \ref{goodexist}.}
	\label{fig: d23}
\end{figure}

{\bf Case D.2.3.1:} Suppose $C_1$ and $C_2$ are both outward pentagons.  Then $C_1$ and $C_2$ are nonadjacent since, by assumption, each outward pentagon is adjacent to two inward pentagons.  In particular, the face $D$ is in fact an inward pentagon lying on $S$.  

Therefore S has exactly six pentagons lying on it, exactly three of which are outward pointing.  Consider the three edges of $X$ emanating from vertices of these outward pointing pentagons.  
If any two of these edges share an endpoint, then these are edges of a pentagon lying in $X$ which is a contradiction.  Therefore, these three edges have distinct endpoints.  It is evident then that they form a prismatic 3--circuit in $P$ which is a contradiction.
\smallskip

{\bf Case D.2.3.2:} Suppose $C_1$ and $C_2$ are both inward pentagons.  Then the faces $C_1$ and $C_2$ cannot be adjacent.  For, if they were, then the single face in $X$ which is adjacent to $C_1$, $C_2$, $B_1$, $B_2$ and $A$ would be adjacent to itself in the edge emanating from the vertex of $A$ lying on $S$.

Therefore there is at least one more pentagon $E$ adjacent to, say, $C_1$.  Suppose $E$ is outward and adjacent to $C_2$.  Let $e$ denote the edge emanating from $E$ into $X$.  Let $a$ denote the edge emanating from $A$ into $X$.  If $a = e$, then $X$ contains a pentagon which is a contradiction.  If $a$ and $e$ share an endpoint, then there is a pair of faces in $X$ which intersect in both $a$ and $e$ which is again a contradiction.  

Thus either $E$ is inward or there are more pentagons lying on $S$.  Thus the face $D$ which is adjacent to $E$, $C_1$, $C_2$, $H_1$ and $H_2$ must have at least 6 edges and so is large.  Since the edge $C_1 \cap E$ edge connects $F$ to a face in $X$, this edge is good.
\smallskip

{\bf Case D.2.3.3:} Suppose exactly one of $C_1$ and $C_2$ is outward.  Without loss of generality, suppose $C_1$ is outward and $C_2$ is inward.  Then $C_1$ and $C_2$ cannot be adjacent since if they were, the faces $C_1$, $C_2$, $H_1$ and $D$ would all share a vertex which contradicts trivalence.  So there is another inward pentagon adjacent to $C_1$.  This pentagon must be $D$ by trivalence.  The edge of $D$ which lies on $S$ edge connects $C_1$ and $C_2$ and is incident to the edge shared by $D$ and $C_2$.  This contradicts the assumption that $C_2$ is inward. 

\smallskip

{\bf Case D.3:} Suppose all pentagons lying on $S$ are outward.  Then $D$ is a pentagonal flower $L^n$.  Recall this means that $D$ looks like an $n$--gon surrounded on all sides by a pentagon.  If $n \geq 6$, then the edges which are intersections of adjacent pentagonal petals of $L^n$ are good edges as they edge connect the $n$--gon in $D$ to a face in $X$.

So suppose that $n=5$ so that $D$ is a dodecahedral flower.  Let $G_1, \dots, G_5$ denote the ring of large faces in $X$ indexed cyclically, each of which are adjacent to a pair of pentagons in $D$.  Let $c_i$ denote the edge $G_i \cap G_{i+1}$.  Then there is a simple closed curve $c$ which lies entirely in the union of the $G_i$ which intersects each $c_i$ transversely.  This curve $c$ is a prismatic $5$--circuit.

On the side of $c$ which contains $D$, $c$ bounds a pentagon in each $G_i$.  Since each $G_i$ is large, in each $G_i$, $c$ bounds a polygon on the other side which has at least 5 edges.  Therefore, by Theorem \ref{p5cdec}, $P$ is decomposable along $c$.  The component of the decomposition which contains $D$ is a dodecahedron.

This concludes the proof of Theorem \ref{goodexist}.
\end{proof}

\begin{theorem}\label{vgcomp} If $P$ a right-angled hyperbolic polyhedron not of L\"obell type, then either it has a very good edge or it is decomposable.
\end{theorem}

\begin{proof} Suppose $P$ is not decomposable.  By Theorem \ref{goodexist}, $P$ has a good edge.  The only way for this edge to fail to be a very good edge is if it is a member of some prismatic $5$--circuit in $P$.  Some facts about prismatic $5$--circuits will need to be collected.

Suppose $c = \{e_1, \dots, e_5\}$ is a prismatic $5$--circuit with indices taken cyclicly.   Recall that this means that there is a simple closed curve $c$ in $\partial P$ which intersects the $1$--skeleton of $P$ exactly in $e_i$ and does so transversely, and furthermore all vertices of these edges are distinct. Then $e_i$ and $e_{i+1}$ are edges contained in some face which will be denoted $F_i$.  This ring of faces $\{F_i\}$ of $P$ is called the set of faces the prismatic $5$--circuit $c$ intersects.  See Figure \ref{fig: p5c}.

\begin{figure}
	\centering
	\includegraphics[bb=0 0 340 104]{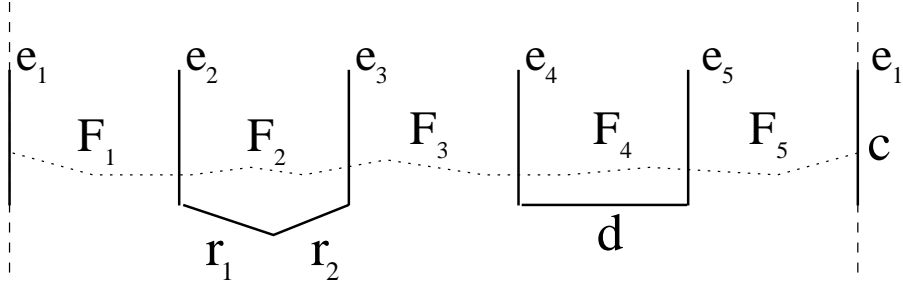}
	\caption{A prismatic 5--circuit and the faces it intersects. As usual, identify the dashed lines.  The edge $d$ is a flat of the prismatic circuit $c$, while edges $r_1$ and $r_2$ form a roof of $c$.}
	\label{fig: p5c}
\end{figure}

\begin{lemma} $F_i \neq F_j$ for $i\neq j$.
\end{lemma}

\begin{proof} Suppose for contradiction that $F_i = F_j$ and without loss of generality suppose $i=1$.  By transversality when $c$ intersects an edge of $F_1$, it must immediately leave $F_1$, therefore $j \neq 2,5$.  Suppose then $j=3$ or $j=4$.  By turning the labelling around only the case when $j=3$ needs to be checked.  In this case, $F_2$ intersects $F_1 = F_3$ in edges $e_1$ and $e_2$ which are distinct.  This is a contradiction.
\end{proof}

Call a prismatic $5$--circuit {\it trivial} if every edge has exactly an endpoint belonging to a fixed pentagonal face and {\it nontrivial} if otherwise.  Note that none of the edges belonging to a trivial prismatic $5$--circuit can possibly be good.  

Call an edge $d$ of $F_i$ a {\it flat} of the prismatic 5--circuit $c$ if $d$ and $c \cap F_i$ are opposite edges of some combinatorial quadrilateral in $F_i$.  For example, the edge $d$ in Figure \ref{fig: p5c} is a flat of the prismatic circuit $c$. 

\begin{lemma} A nontrivial prismatic $5$--circuit $c$ in a right-angled hyperbolic polyhedron cannot have five adjacent flats on the same side of $c$.
\end{lemma}
\begin{proof} By trivalence, there must be a single face adjacent to all these flats.  Evidently, this face must be a pentagon and the prismatic $5$--circuit must be trivial.
\end{proof}

\begin{lemma} A nontrivial prismatic $5$--circuit $c$ in a right-angled hyperbolic polyhedron cannot have four adjacent flats on the same side of $c$.  
\end{lemma}
\begin{proof} Without loss of generality, suppose for a contradiction that $F_1, \dots F_4$ have flat edges on the same side of $c$.  There is a face $G$ which is adjacent to each of $F_1, \dots F_4$.  Note that $F_5$ cannot have a flat on the same side of $c$ as the other flats by the nontriviality of $c$.  Then $G$ intersects $F_5$ in at least two edges, which contradicts the combinatorics of a polyhedron.
\end{proof}

\begin{lemma} A nontrivial prismatic $5$--circuit $c$ in a right-angled hyperbolic polyhedron cannot have three adjacent flats on the same side of $c$.  
\end{lemma}
\begin{proof} Without loss of generality, suppose for a contradiction that $F_1, F_2, F_3$ have flat edges on the same side of $c$.  There is a face $G$ which is adjacent to each of these.  By trivalence, $G$ also intersects each of $F_4$ and $F_5$.  Consider the sequence of three edges ${d_i}$ given by
$$d_1 = G \cap F_5, d_2 = G \cap F_4, d_3 = F_4 \cap F_5.$$
Note that neither $d_1$ nor $d_2$ can share an endpoint with $d_3$ since this would imply four adjacent flats.  

Suppose $d_1$ and $d_2$ share an endpoint so that $G$ is a pentagon.  Let $q$ be the edge which shares an endpoint with $d_1$ and $d_2$ but is not $d_1$ nor $d_2$.  By the above, $q$ cannot be $d_3$ either.  However, $q$ is an edge of both $F_4$ and $F_5$ which implies these faces intersect in at least two edges which is a contradiction.

Therefore $d_1$, $d_2$, and $d_3$ form a prismatic $3$--circuit which is a contradiction.
\end{proof}

\begin{lemma}\label{adjflat} A nontrivial prismatic $5$--circuit $c$ in a right-angled hyperbolic polyhedron cannot have adjacent flats on the same side of $c$.  
\end{lemma}

\begin{proof} Without loss of generality, suppose $F_1$ and $F_2$ have flat edges of the same side of $c$.  Then there is a face $G$ which is adjacent to both $F_1$ and $F_2$.  By trivalence, $G$ is also adjacent to $F_5$ and $F_3$.  Consider the sequence of four edges ${d_i}$:
$$d_1 = F_5 \cap G, d_2 = G \cap F_3, d_3 = F_3 \cap F_4, d_4 = F_4 \cap F_5.$$
See Figure \ref{fig: adjflat}.

Note that $d_3$ and $d_4$ are contained in the prismatic $5$--circuit $c$ and therefore must have distinct endpoints.  Note also that $d_1$ and $d_2$ have distinct endpoints since the face $G$ must have at least $5$ vertices.  

\begin{figure}
	\centering
	\includegraphics[bb=0 0 340 104]{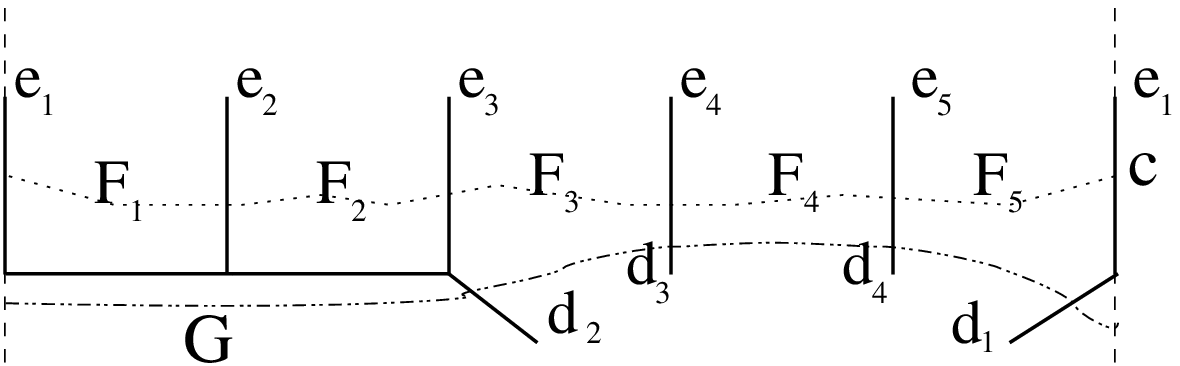}
	\caption{Lemma \ref{adjflat}.  }
	\label{fig: adjflat}
\end{figure}

Suppose $d_1$ and $d_3$ share an endpoint.  Since $d_3$ is an edge of $F_3$ and $F_4$, by trivalence and the property of polyhedra that every edge belong to exactly two faces, $d_1$ must also be an edge of $F_3$ or $F_4$.  But $d_1$ is an edge of the faces $G$ and $F_5$.  The latter, by the above lemma, cannot be $F_3$ or $F_4$.  Evidently, $G$ is not $F_3$ and cannot be $F_4$ since this would imply that $d_2 = d_3 = e_3$ and therefore $e_2$ and $e_3$ share an endpoint.  This is a contradiction and so $d_1$ and $d_3$ are disjoint.  A similar argument by a relabelling of edges shows $d_2$ and $d_4$ share no endpoints.

Suppose $d_1$ and $d_4$ share an endpoint.  Then $d_1$ is a flat of the face $F_5$ and so there are three adjacent flats which, by the above lemma, is a contradiction.  A similar argument shows $d_2$ and $d_3$ cannot share an endpoint.

Therefore, $d_1, d_2, d_3, d_4$ all have distinct endpoints and therefore form a prismatic $4$--circuit in $P$, which is a contradiction.
\end{proof}

For a prismatic $k$--circuit $c$, define a {\it roof} of $c$ to be a pair of edges $r_1$ and $r_2$ in a particular face $F$ which $c$ intersects such that the the three segments $r_1, r_2$, and $F \cap c$ together bound a combinatorial pentagon in $F$.  For example, the edges $r_1$ and $r_2$ in Figure \ref{fig: p5c} form a roof of the prismatic circuit $c$.

\begin{lemma}\label{frf} A prismatic 5--circuit $c$ in a right-angled hyperbolic polyhedron cannot have a pair of flats on the same side of $c$ separated by a single roof.
\end{lemma}
\begin{proof} For the purposes of establishing a contradiction, and without loss of generality, suppose $F_1$ and $F_3$ contain flat edges on the same side of $c$, and $F_2$ contains a roof of $c$ separating these flats.  Then adjacent to each of $F_1$ and $F_3$ along their flats are faces $G_1$ and $G_3$ which are themselves both adjacent to $F_2$ by Lemma \ref{verteximpliesadjacent} as they evidently share a vertex with it.  Note also that $G_1$ is adjacent to $F_5$ and $G_3$  is adjacent to $F_4$ for similar reasons.  

Let $g_1 = G_1 \cap F_2$ and $g_3 = G_3 \cap F_2$.  These edges $g_1$ and $g_3$ evidently form the roof of $c$ contained in $F_2$, and therefore are adjacent.  Therefore, $G_1$ and $G_3$ are themselves adjacent, again, by Lemma \ref{verteximpliesadjacent}. 

Consider then a curve which passes through the edges $$d_1 = G_1 \cap F_5, d_2 = G_1 \cap G_3, d_3 = G_3 \cap F_4, d_4 = F_4 \cap F_5.$$  If $d_1$ shares an endpoint with $d_4$, then $d_1$ is a flat of the face $F_5$.  This implies $c$ has a pair of adjacent flats on the same side which is a contradiction to Lemma \ref{adjflat}.  A similar argument shows $d_3$ and $d_4$ do not share an endpoint.  

Furthermore, $d_1$ and $d_2$ cannot share endpoints as the face $G_1$ must have at least five edges.  Similarly, $d_2$ and $d_3$ cannot share endpoints.  

If $d_1$ and $d_3$ share an endpoint, then $F_4$ and $F_5$ are adjacent in some edge other than $e_5 = d_4$.  This is a contradiction.  Similarly, if $d_2$ and $d_4$ share an endpoint, then $G_1$ and $F_5$ intersect in a vertex not in $d_1$ as $d_2$ and $d_1$ were shown to be disjoint.  This is also a contradiction.

Therefore, $d_1, d_2, d_3, d_4$ all have distinct endpoints and so form a prismatic 4--circuit which is a contradiction.  This proves Lemma \ref{frf}.
\end{proof}

\begin{lemma}\label{gflatvg} If a flat $e$ of a prismatic $5$--circuit $c$ is a good edge, then it cannot be a part of a prismatic $5$--circuit and so is very good.
\end{lemma}
\begin{proof} For the purposes of establishing a contradiction, suppose that $e$ is an edge of a prismatic $5$--circuit $d$.  Label the edges of $c$ by $\{c_i\}$, the edges of $d$ by $\{d_i\}$, the faces that $c$ intersects $\{F_i\}$ and the faces that $d$ intersects $\{G_i\}$.  Without loss of generality, suppose $e$ is an edge of $F_1$, suppose $G_1 = F_1$ and suppose $G_2$ is the face adjacent to $G_1$ in $e$.  

Note that since $e$ is a flat of $c$, it edge connects $F_5$ and $F_2$.  Since $e$ is a good edge, $F_5$ and $F_2$ are large faces. 

\begin{sublemma} Either $G_3$ or $G_4$ is either $F_3$ or $F_4$.
\end{sublemma}

\begin{proof} Consider the simple closed curves $c$ and $d$.  These curves intersect in the face $F_1 = G_1$ since the edge $e$ is a flat of $c$, which means $d$ must intersect a pair of edges which lie on opposite sides of $c$, one of which is $e$.  Note that an isotopy can be performed so that all intersections between $c$ and $d$ occur in the interiors of faces of $P$ and within any face, $c$ and $d$ intersect at most once.  Since $\partial P$ is topologically $S^2$, these curves must intersect at least twice.  This implies that some $G_i$ is equal to some $F_j$ for $i,j \neq 1$.  

$G_5$ cannot be $F_2$ or $F_5$ since $G_1 \cap F_2 = c_1$ and $G_1 \cap F_5 = c_5$ both share an endpoint with $e$.  Also $G_5$ cannot be $F_3$ or $F_4$.  For example, if $G_5$ is $F_3$, then $F_1 = G_1$ is adjacent to $F_3 = G_5$ in an edge which either shares a vertex with $c_2$ or does not.  In the former case, $F_1$, $F_2$ and $F_3$ for pairwise adjacent faces and so must share a vertex which contradicts $c$ being a prismatic 5--circuit.  In the latter case, $c_1$, $c_2$ and $F_1 \cap F_3$ form a prismatic 3--circuit which is also a contradiction.  A similar proof shows $G_5$ cannot be $F_4$.  Therefore, $G_5$ is not $F_i$ for any $i$.  

As $e = G_2 \cap F_1$ is a flat of $c$, $G_2$ cannot be $F_2$ or $F_5$.  $G_2$ also cannot be $F_3$ or $F_4$.  For if, for example, $G_2$ was $F_3$, then $F_1$, $F_2$, and $F_3$ would be pairwise adjacent faces and so must share a vertex which contradicts $c$ being a prismatic 5--circuit.  Therefore, $G_5$ is not $F_i$ for any $i$.

Note also $G_3$ cannot be $F_5$ or $F_2$ as this would imply that $G_1$, $G_2$ and $G_3$ are pairwise adjacent faces and thus must share a vertex which contradicts $d$ being a prismatic 5--circuit.  Similarly $G_4$ cannot be $F_5$ or $F_2$.

Therefore, $G_3$ or $G_4$ is $F_3$ or $F_4$.  This concludes the proof of the sublemma.
\end{proof}

Suppose first that $G_3$ is $F_3$.  Note that as $G_2$ and $F_2$ are distinct faces which share a vertex, namely $e \cap c_1$, by Lemma \ref{verteximpliesadjacent} $G_2$ and $F_2$ are adjacent.  Furthermore, by definition, $G_3 = F_3$ is adjacent to $F_2$ and $G_2$, and so these three faces all share a vertex.  Therefore the edge $G_2 \cap F_2$ has endpoints $e \cap c_1$ and an endpoint of $c_2$ and is thus a flat of $c$.  But this implies $c$ has a two adjacent flats on the same side which contradicts Lemma \ref{adjflat}.

A similar argument shows $G_3$ cannot be $F_4$.

Suppose next that $G_4$ is $F_3$.  Then $G_5$ is adjacent to both $G_1 = F_1$ and $G_4 = F_3$.   Furthermore $F_2$ is adjacent to $F_1$ and $F_3$, and, as the proof of the above sublemma shows, $G_5$ is not $F_2$.  Therefore, by Lemma \ref{fouradj}, $G_5$ and $F_2$ are adjacent.

Now consider $G_3$.  This face is adjacent to $G_2$ and $G_4 = F_3$.  Furthermore $F_2$ is also adjacent to $G_2$ and $G_4$, and, as the proof of the above sublemma shows, $G_3$ is not $F_2$.  Therefore, again by Lemma \ref{fouradj}, $G_3$ and $F_2$ are adjacent.  

Therefore, for each $i$, $G_i$, $F_2$, and $G_{i+1}$ and pairwise adjacent and thus share a vertex.  This evidently implies $F_2$ must be a pentagon.  This is a contradiction as $e$ was assumed to be a good edge.  Therefore $G_4$ cannot be $F_3$.  

A similar argument shows $G_4$ cannot be $F_4$.

Therefore, the good flat $e$ must be a very good edge.  This concludes the proof of Lemma \ref{gflatvg}.  
\end{proof}

An obvious result which will be used often is the following:

\begin{lemma}\label{pentflat} Suppose $F_i$ is a pentagon.  Then $c$ has a flat which is an edge of this pentagon.  Furthermore, $c$ has a roof formed by a pair of adjacent edges of this pentagon.  
\end{lemma}

\begin{proof} The curve $c$ intersects two nonadjacent edges of $F_i$.  The remaining three edges must be distributed so that exactly two are on one side of $c$ and one on the other side.  The former two edges are a roof and the latter edge is a flat.  
\end{proof}

At last, returning to the proof of Theorem \ref{vgcomp}, suppose that $P$ has a good edge $e$ which is a member of a prismatic $5$--circuit $c$.  Then $c$ cannot be a trivial prismatic circuit.  

The proof of Theorem \ref{vgcomp} now breaks up into six cases with possibly some subcases depending on the number of large faces the prismatic 5--circuit $c$ intersects.  All but one of these cases either produces a contradiction to the combinatorial facts about prismatic 5--circuits collected above, or produces a very good edge via Lemma \ref{gflatvg}.  The remaining case implies that $P$ is decomposable.  

As usual, let $F_i$ be the face that contains both the edges $c_i$ and $c_{i+1}$.
\smallskip

{\bf Case A:}  Supppose no $F_i$ is large.  Then, by Lemma \ref{pentflat} $c$ has five flats, one for each pentagon $F_i$.  There is no possible arrangement of these pentagons for which there are no adjacent pairs of flats on the same side of $c$.  This is a contradiction to Lemma \ref{adjflat}.
\smallskip

{\bf Case B:}  Suppose exactly one of $F_i$ is large.  Then by Lemma \ref{pentflat}, each of these pentagons have an edge which is a flat of $c$ on one side, and a roof of $c$ on the other.  Because there cannot be adjacent flats on the same side of $c$, there is only way for the pentagons to be arranged up to isomorphism.  That is with flats ``alternating'' on either side of $c$.  Then on either side of $c$, there is a flat, then roof, then flat, then roof coming from the edges of the string of four pentagons.  This contradicts Lemma \ref{frf}.
\smallskip

{\bf Case C:}  Suppose exactly two of the $F_i$ are large.  Then either the three pentagons are arranged in a row, or they are not.  If they are not, then there is a pentagon whose neighbors are the large faces in $\{F_i\}$.  These faces are edge connected by the flat of the pentagon described by Lemma \ref{pentflat} and therefore is a very good edge by \ref{gflatvg}.

Suppose then that the three pentagons are arranged in a row.  Then since each pentagon has a flat on one side of $c$ and a roof on the other by Lemma \ref{pentflat}, either there are two adjacent flats on the same side of $c$ which contradicts Lemma \ref{adjflat}, or there is a flat, then roof, then flat on the same side of $c$ which contradicts Lemma \ref{frf}.
\smallskip

{\bf Case D:}  Suppose exactly three of the $F_i$ are large.  Then either the two pentagons are adjacent or they are not.  Each case will be treated separately:
\smallskip

{\bf Case D.1:} If they are not adjacent, then the flat edge of a pentagon described by Lemma \ref{pentflat} edge connects two large faces and so, by Lemma \ref{gflatvg}, is very good.  
\smallskip

{\bf Case D.2:} Suppose the two pentagons are adjacent.  Without loss of generality, suppose $F_2$ and $F_3$ are the pentagons.  Then $F_1$ and $F_4$ are large faces.  Make a choice of side of $c$ by choosing the side which contains the flat of the pentagon $F_2$.   Let $n_c$ denote the number of edges of $F_1$ and $F_4$ contained entirely in this chosen side of $c$.  Note that this number $n_c$ must be at least 2 as $c$ is prismatic.

Suppose $n_c = 2,3$.  Then one of $F_1$ or $F_4$ contains a flat of $c$ on the chosen side.  If $F_1$ is the culprit, then $c$ has a pair of adjacent flats on the chosen side which contradicts Lemma \ref{adjflat}.  If $F_4$ is the culprit, then since $F_3$ must contain a roof of $c$ on the chosen side, there is a contradiction to Lemma \ref{frf}.  Thus $n_c \geq 4$ and, in particular, the contribution from each of $F_1$ and $F_4$ to $n_c$ must be at least two edges.

Label the face of $P$ adjacent to $F_2$ through the edge which is a flat of $c$ by $G$.  Then $G$ is adjacent to both $F_1$ and $F_3$.  Note also that $G$ is edge connected to the face $F_4$ by an edge of $F_3$.  Call this edge $g$.  Consider the curve $d$ which passes through the edges: $$d_1 = F_1 \cap F_5, d_2 = F_1 \cap G, d_3 = G \cap F_3, d_4 = F_3 \cap F_4, d_5 = F_4 \cap F_5.$$

This curve $d$ forms a prismatic 5--circuit.  Furthermore, the edge $g$ is a flat of $d$.  Thus, if $G$ is a large face, then $g$ is a good edge which is also a flat of a prismatic 5--circuit and so is very good by Lemma \ref{gflatvg}.  

So suppose $G$ is a pentagon.  Then the prismatic 5--circuit $d$ satisfies the conditions of Case D.2 and furthermore, by choosing the side of $d$ which contains the flat $g$ one has $n_d = n_c - 1$ with the contribution of $F_1$ to $n_d$ being one less than that to $n_c$.  See Figure \ref{fig: case8.3d}.

\begin{figure}
	\centering
	\includegraphics[bb=0 0 337 125]{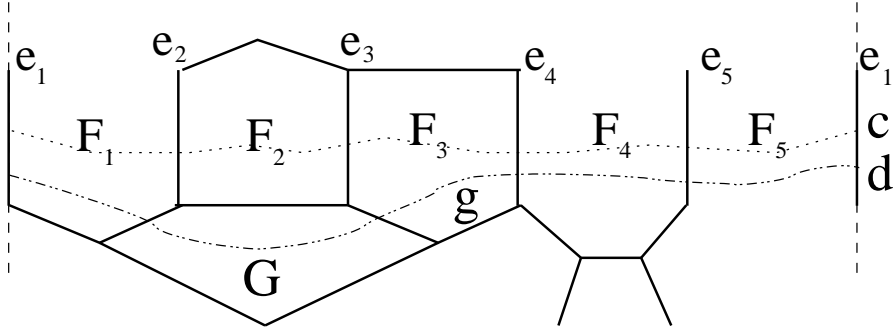}
	\caption{Case D.2 of Theorem \ref{vgcomp}.  The prismatic 5--circuit $c$ has $n_c = 5$.  In this figure, as $G$ is a pentagon, the prismatic circuit $d$ contradicts Lemma \ref{frf}.}
	\label{fig: case8.3d}
\end{figure}

Now an argument by induction on these numbers $n$ can proceed.  At each step, either a very good edge is exhibited, or a new prismatic 5--circuit is produced with smaller $n$.  If $n$ gets too small, that is $n = 3$, or the contribution from the faces $F_1$ or $F_4$ to $n$ is too small, that is a single edge, then the above argument gives a contradiction.  

\smallskip

{\bf Case E:} Suppose exactly four of the $F_i$ are large.  Without loss of generality, suppose $F_1$ is a pentagon.  Then by Lemma \ref{pentflat}, $F_1$ has an edge $f$ which is a flat of $c$.  As it edge connects $F_5$ with $F_2$ which are large faces, this edge $f$ is good.  By Lemma \ref{gflatvg}, this edge is very good.
\smallskip

{\bf Case F:} Suppose all $F_i$ are large.  Then either $c$ has a flat or it doesn't.  If $c$ has a flat, then this flat evidently edge connects two large faces $F_j$ and $F_{j+2}$, and thus by Lemma \ref{gflatvg}, is very good.  If, on the other hand, $c$ does not have a flat, then $P$ is decomposable by Theorem \ref{p5cdec}, which is a contradiction.

This concludes the proof of Theorem \ref{vgcomp}.
\end{proof}

\section{Geometric Realization of Edge Surgery}
Edge surgery as defined is a purely combinatorial operation.  It does, however, have a geometric realization as a cone manifold deformation.  Intuitively, an edge surgery looks like ``unbending'' along the edge until the supporting planes of the faces adjacent to the edge coincide.  That is, the dihedral angle of the edge being surgered will increase from $\frac{\pi}{2}$ to $\pi$ while all other edges retain their right-angles.

This section will be devoted to showing that this deformation is a path through the space of hyperbolic polyhedra which begins and ends at right-angled polyhedra.  By use of the Schl\"afli differential formula, it will be shown that the volume of the initial right-angled polyhedron is greater than that of the final right-angled polyhedron.

To prove this deformation is through hyperbolic polyhedra, a generalization of Andreev's Theorem that accounts for obtuse angled hyperbolic polyhedra will be required.  I Rivin in his thesis accomplished such a generalization.  This result is communicated by a paper of Rivin and Hodgson \cite{RH}.  

Let $P \subset \HH^3$ be a hyperbolic polyhedron, not necessarily non-obtuse angled.  Define the {\it spherical polar} $v^*$ of a vertex $v$ of $P$ by associating to each vertex $v$ the set of outward unit normal vectors to planes which are incident to $v$ but are disjoint from the interior of $P$ (the {\it support planes} to $P$ at $v$).  Then $v^*$ is a spherical polygon whose edge lengths are the exterior dihedral angle measures of edges incident to $v$ in $P$.  It follows that the interior angle measures of $v^*$ are the complementary angle measures of the face angles of $P$ at $v$. 

Define the {\it spherical polar} $P^*$ of $P$ to be the piecewise spherical metric space constructed by gluing spherical polygons $v^*$ and $w^*$ associated to a pair of vertices $v$ and $w$ of $P$ exactly when $v$ and $w$ are connected by an edge.  Note that this metric space $P^*$ is topologically $\bS^2$ and is combinatorially dual to the cell structure of $P$ (or, more properly speaking, the boundary of $P$).  However, $P^*$ will often not be isometric to $\bS^2$.

To illustrate this point by way of example, suppose $P$ is a right-angled polyhedron.  Then the spherical polar of each vertex of $P$ is a right-angled triangle in $\bS^2$.  Since every face of $P$ has at least 5 edges, around every vertex in $P^*$ are 5 or more of these right-angled triangles, and in particular, the sum of the angles around such a vertex is greater than $2\pi$.  Thus the vertices of $P^*$ are singularities in the spherical metric.  They represent accumulations of negative curvature.

Such singular points where the curvature is not 1 are called {\it cone singularities}.  The {\it cone angle} of such a singularity is the sum of the angles around the singularity, or more intrinsically, the length of the link of the singularity viewed as a piecewise circular metric space.  A cone angle of $2\pi$ corresponds to a nonsingular point.

I Rivin's generalization of Andreev's Theorem \cite{RH} characterizes those metric spaces homeomorphic to $\bS^2$ which arise as polars of hyperbolic polyhedra:

\begin{theorem} \label{Rivin} A metric space $Q$ homeomorphic to $\bS^2$ is the spherical polar of a compact convex polyhedron $P$ in $\HH^3$ if and only if each of the following conditions hold:

\setcounter{enumi}{0}
\begin{enumerate}
\item $Q$ is piecewise spherical with constant curvature 1 away from a finite collection of cone points $c_i$.

\item The cone angles at the points $c_i$ are greater than $2\pi$.

\item The lengths of closed geodesics are all greater than $2\pi$.
\end{enumerate}
The metric space $Q$ determines the polyhedron $P$ completely up to isometry.

\end{theorem}

The word ``geodesic'' in the statement of condition (3) will be taken to mean a locally distance minimizing path/loop.  A geodesic in a piecewise spherical metric space is made up of arcs of great circles in cells such that at any point along the path of the geodesic, the angle subtended has measure greater than or equal to $\pi$ on either side.

Suppose that $P_0$ is a right-angled hyperbolic polyhedron with a very good edge $e$.  Think of $P_0$ combinatorially as the underlying combinatorial polyhedron, together with a labelling of each edge by the dihedral angle measure in its geometric realization.  In this case, each edge of $P_0$ is labelled by $\frac{\pi}{2}$.  

For each $t \in [0,1)$, let $P_t$ denote the combinatorial polyhedron isomorphic to the combinatorial polyhedron underlying $P_0$, every edge other than $e$ labelled by $\frac{\pi}{2}$, and the edge $e$ labelled by $\theta_t = (1-t) \frac{\pi}{2} + t\pi$.  Let $P_1$ denote the combinatorial polyhedron obtained by edge surgery of $P_0$ along the edge $e$ with each edge labelled by $\frac{\pi}{2}$.  

\begin{theorem} \label{P_t} Each $P_t$ has a geometric realization as a hyperbolic polyhedron.
\end{theorem}

\begin{proof} By Theorem \ref{edgesurgeryrightangled} and assumption, $P_0$ and $P_1$ have geometric realizations as hyperbolic polyhedra (in fact, they are right-angled).  So assume $t$ lies in the open interval $(0,1)$, so that $P_t$ is combinatorially isomorphic to $P_0$.  

Let $Q_t$ denote the piecewise spherical metric space constructed in the following way.  For each vertex $v$ of $P_t$, construct a spherical triangle whose edge lengths are the complementary angle measures of the dihedral angle measures of the edges of $P_t$ incident to $v$.  Identify the edges of two such triangles if and only if their associated vertices are connected by an edge of $P_t$.  Then $Q_t$ is evidently a metric space which is homeomorphic to $\bS^2$.  It will be shown that $Q_t$ satisfies the conditions of Theorem \ref{Rivin} implying that $Q_t$ is the spherical polar of a hyperbolic polyhedron.  The theorem will be proved when it is shown that the hyperbolic polyhedron in question is in fact the geometric realization of $P_t$ by controlling for the combinatorics of the spherical cell division of $Q_t$.  

To show condition (1) of Theorem \ref{Rivin} holds for $Q_t$, note that the singular points of $Q_t$ correspond to the faces of $P_t$ since every other point lies in the interior of a face or the interior of an edge where two triangles meet.  These points have curvature 1.  Since $P_t$ has only finitely many faces, there are only finitely many cone points, thus demonstrating the first condition.  

The cone angle of a singular point $c_i$ is the sum of the angle measures at $c_i$ of the spherical triangles incident to $c_i$.   Suppose $F_i$ is the face of $P_t$ corresponding to $c_i$.  If $F_i$ is disjoint from the edge $e \subset P_t$, then each triangle incident to $c_i$ is isometric to a right-angled spherical triangle.  As $F_i$ has at least $k \geq 5$ edges, the cone angle at $c_i$ is $k \frac{\pi}{2} > 2\pi$.  

Suppose that $F_i$ contains the edge $e$ and let $k \geq 5$ be the number of edges of $F_i$.  Then $c_i$ is incident to $k-2$ right-angled triangles, and two triangles whose lengths are $\frac{\pi}{2}$, $\frac{\pi}{2}$, and $\pi - \theta_t$.  A bit of elementary spherical geometry reveals that such a triangle has interior angle measures equal to the length of the edge opposite the angle.  Therefore, these two triangles are incident to $c_i$ in right angles, and so the cone angle around $c_i$ is $k \frac{\pi}{2} > 2\pi$.  

Suppose that $F_i$ contains a single vertex of $e$.  Then $e$ edge connects $F_i$ to some other face of $P_t$, and thus, by assumption, is very good.  This implies that the face $F_i$ is large and so has $k \geq 6$ edges.  Thus $c_i$, has $k \geq 6$ triangles incident to it.  All of the triangles incident to $c_i$ except one are right-angled.  The exceptional triangle is isometric to the sort of triangle described above whose lengths are $\frac{\pi}{2}$, $\frac{\pi}{2}$, and $\pi - \theta_t$.  This triangle is incident to $c_i$ in the non-right angle, and therefore the cone angle of this point is given by $(k-1)\frac{\pi}{2} + \pi - \theta_t > 2\pi$.  This shows condition (2).

Denote the edge of $Q_t$ which is dual to $e \subset P_t$ by $e^*$.  This edge $e^*$ is contained in a pair of triangles $T_1$ and $T_2$ isometric to the ones described above with edge lengths $\frac{\pi}{2}$, $\frac{\pi}{2}$, and $\pi - \theta_t$.  Let $B_t$ denote the spherical bigon which is the union of these two triangles.  Note that the two edges of $B_t$ both have length $\pi$ which meet at two points at an angle whose measure is $\pi - \theta_t$.  Denote the endpoints of $e^*$ by $w_1$ and $w_2$ and denote the vertices of $B_t$ which are not endpoints of $e^*$ by $v_1$ and $v_2$.  See Figure \ref{fig: dualedge}.

\begin{figure}
	\centering
	\includegraphics[bb=0 0 330 146]{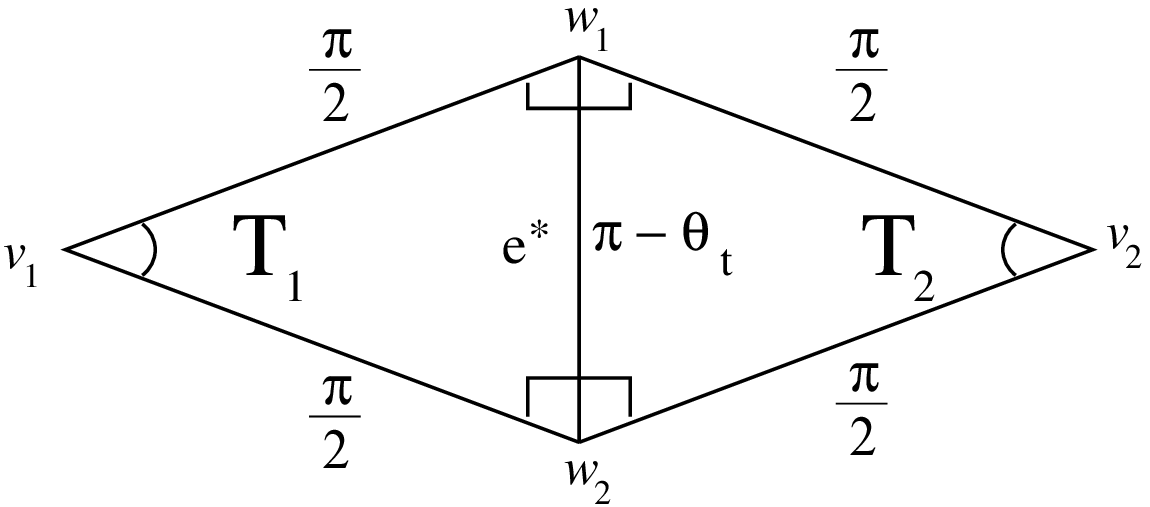}
	\caption{The bigon $B_t$ in the spherical dual of $P_t$.}
	\label{fig: dualedge}
\end{figure}

Note that if a point in $B_t$ is distance $d$ from $v_1$, then it is distance $\pi - d$ from $v_2$.  Therefore, $B_t$ is foliated by arcs of circles of radius $d$ from the point $v_1$ where $d \in [0, \pi]$, and the leaf space of this foliation is isometric to the interval $[0, \pi]$.   

Note also that as $t$ approaches 1, the bigon $B_t \subset Q_t$ degenerates into a line segment $B_1$ isometric to $[0,\pi]$.

Fix $t \in (0,1)$ and let $\overline{f} : B_t \rightarrow [0,\pi] \cong B_1$ denote the projection onto the leaf space.  Let $f : Q_t \rightarrow Q_1$ be the continuous map which when restricted to any cell of $Q_t$ which does not belong to $B_t$ realizes the natural correspondence of cells by an isometry, and which when restricted to $B_t$, is the map $\overline{f}$.

\begin{lemma} \label{bigon} This map $f$ is distance nonincreasing and is a local isometry when restricted to $Q_t \setminus B_t$.  In particular, it is an isometry when restricted to the star of any vertex of $Q_t$ not contained in the bigon $B_t$.
\end{lemma}

\begin{proof} Let $x$ and $y$ be a pair of points of $Q_t$ and let $\gamma$ be a geodesic segment connecting them whose length is the distance between them.  If $\gamma$ does not intersect the bigon $B_t$, then $f(\gamma)$ is a geodesic segment in $Q_1$ and so $d_t(x,y) = d_1(f(x),f(y))$ where $d_t$ and $d_1$ denote the metrics in $Q_t$ and $Q_1$ respectively.

So suppose the geodesic segment $\gamma$ intersects the bigon $B_t$, and let $\tau = \gamma \cap B_t$.  Denote the endpoints of $\tau$ in $B_t$ by $a$ and $b$.  Then $f(\gamma)$ is a broken geodesic with a segment given by $f(\tau)$ lying on the interval $B_1 = [0, \pi]$ with endpoints given by $f(a), f(b) \in [0,\pi]$.  

Without loss of generality, assume that $f(a) \geq f(b)$.  Then the length of $f(\tau)$ is $f(a) - f(b)$.  Since $f$ is a local isometry outside of $B_t$, to prove the lemma it will be shown that the length of $\tau$ in $Q_t$, denoted $l(\tau)$, is greater than or equal to $f(a)-f(b)$.  

If $f(a)-f(b) > l(\tau)$, then in $B_t$, the path $\tau$ concatenated with the geodesic segment connecting $b$ to $v_1$, running along the boundary, of $B_t$ has endpoints $a$ and $v_1$ and has length $l(\tau) + f(b)$ which is strictly less than $f(a)$.  However, the distance between $a$ and $v_1$ is given by $f(a)$.  This is a contradiction.
\end{proof}

If $v$ is a vertex of $Q_t$, let $st(v)$ denote the {\it star} of $v$.  Let $ost(v)$ denote the {\it open star} of $v$, which is the interior of $st(v)$.  The following well known and important lemma gives an estimate on the length of a geodesic arc contained in $st(v)$.  For a proof, see \cite{CD}.

\begin{lemma} \label{geodesicinstar} Suppose $Q$ is a piecewise spherical metric space which is the spherical polar of a non-obtuse hyperbolic polyhedron.  Let $v$ be a vertex of $Q$ and $\gamma$ a geodesic segment in $st(v)$ joining two points of $\partial st(v)$ such that $\gamma \cap ost(v) \neq \emptyset$.  Then the length of $\gamma$ is at least $\pi$.
\end{lemma}

Let $\gamma$ now denote a closed geodesic in $Q_t$.  If $\gamma$ does not intersect $int(B_t)$, then $f(\gamma)$ is a closed geodesic in $Q_1$ and so therefore has length greater than $2\pi$ by Rivin's Theorem \ref{Rivin} as $Q_1$ is the spherical polar of the right-angled hyperbolic polyhedron $P_1$.  Therefore, the length of $\gamma$ is also greater than $2\pi$ by Lemma \ref{bigon}.

So suppose $\gamma$ intersects $int(B_t)$.  It is clear that $\gamma$ cannot be completely contained in $B_t$.  Let $\gamma_i \subset \gamma$ denote the closure of a component of $\gamma \cap int(B_t)$.

Fixing notation, let $T_1$ be the triangle in $B_t$ with vertices $v_1$, $w_1$, and $w_2$, and $T_2$ the triangle with vertices $v_2$, $w_1$ and $w_2$.   

Suppose first that $\gamma_i$ contains contains $v_1$ or $v_2$.  Then a bit of elementary spherical geometry implies that $\gamma_i$ contains the other vertex $v_2$ or $v_1$ and that the length of $\gamma_i$ is $\pi$.  Let $X$ and $Y$ denote the right-angled triangles $\gamma$ enters after leaving $B_t$ through $v_1$ and $v_2$ respectively.  Then the lengths of both $\gamma \cap X$ and $\gamma \cap Y$ are $\frac{\pi}{2}$.  Note that if $X$ and $Y$ are adjacent in $Q_t$, then the polyhedron $P_t$ contains a prismatic $4$--circuit which is a contradiction.  Therefore, $X$ and $Y$ are not adjacent, and so the length of $\gamma$ is larger than $2\pi$.

Suppose once and for all that $\gamma_i$ misses $v_1$ and $v_2$.  If $\gamma_i$ misses $e^*$, then it is completely contained in one of the triangles $T_1$ or $T_2$.  Without loss of generality, suppose $\gamma_i \subset T_1$.  Let $X$ and $Y$ denote the right angled triangles adjacent to $T_1$.  let $x$ and $y$ denote the vertices of $X$ and $Y$ respectively which are not contained in $T_1$.  Then $\gamma$ evidently intersects $ost(x)$ and $ost(y)$.  Denote the closure of the component of $ost(x) \cap \gamma$ which is adjacent to $\gamma_i$ by $\gamma_x$ and define $\gamma_y$ similarly.  Then $f(\gamma_x) \subset Q_1$ is a geodesic segment contained in $st(f(x))$ by Lemma \ref{bigon} which intersects $ost(f(x))$ and therefore, by Lemma \ref{geodesicinstar} has length at least $\pi$.  Therefore $\gamma_x$ also has length at least $\pi$.  A similar argument shows $\gamma_y$ has length at least $\pi$ as well.  Note that $ost(x) \cap ost(y) = \emptyset$ since if not, then $P_t$ would contain a prismatic 3 or 4--circuit. Since $\gamma_i$ has nonzero length, $\gamma$ must have length larger than $2\pi$.  

Suppose that $\gamma_i$ intersects $int(e^*)$ transversely.  Let $X$ and $Y$ denote the right-angled triangles adjacent to each of $T_1$ and $T_2$ respectively which contain the endpoints of $\gamma_i$.  A bit of simple spherical geometry shows that $X$ and $Y$  must lie on oppsite sides of $B_t$ in the sense that if $X$ contains $w_1$ say, then $Y$ contains $w_2$.  Let $x$ and $y$ denote the vertices of $X$ and $Y$ respectively which are not contained in $T_1$ and $T_2$ respectively.  Then $\gamma$ evidently intersects $ost(x)$ and $ost(y)$.  By an argument similar to the one above, this implies that the length of $\gamma$ is larger than $2\pi$.

Suppose $\gamma_i$ contains the edge $e^*$.  Then $f(\gamma)$ is a closed geodesic in $Q_1$ and therefore must have length larger than $2\pi$.  Thus by Lemma \ref{bigon}, $\gamma$ must also have length larger than $2\pi$.  

Finally, suppose that $\gamma_i$ intersects $e^*$ exactly in one of its endpoints $w_1$ or $w_2$, and without loss of generality, suppose $w_1 \in \gamma_i$.  Then $\gamma_i$ is contained in either $T_1$ or $T_2$, so suppose also without loss of generality that $\gamma_i \subset T_1$.  Let $X$ denote the right-angled triangle adjacent to $T_1$ which contains the vertices $w_2$ and $v_1$ and let $x$ denote the remaining vertex of $X$.  Then $\gamma$ evidently intersects $ost(x)$.  Let $Y$ denote the right-angled triangle other than $T_1$ that contains the vertex $w_1$ and intersects $\gamma$.  Then $Y$ is contained in the star of some vertex which is not $v_i$ or $w_i$.  Denote this vertex by $y$.  So $\gamma$ intersects both $ost(x)$ and $ost(y)$ and by an argument similar to the one above, $\gamma$ has length larger than $2\pi$.  

This shows that $Q_t$ satisfies condition (3) of Theorem \ref{Rivin} and thus is the spherical polar of some hyperbolic polyhedron.  However, it is too hasty to conclude that this polyhedron is the geometric realization of $P_t$ as it is possible that there is more than one cell decomposition of $Q_t$ which satisfy the conditions of being a spherical polar of a hyperbolic polyhedron.

To prove that this cannot be the case, denote by $\mathcal{C}$ the cell decomposition used to construct $Q_t$ , and denote by $\mathcal{D}$ the cell decomposition of $Q_t$ as the spherical polar of the hyperbolic polyhedron coming from Theorem \ref{Rivin}.  It will be shown that $\mathcal{C} = \mathcal{D}$. 

$\mathcal{D}$ must satisfy the following three conditions: 

\setcounter{enumi}{0}
\begin{enumerate}
\item The vertex set is exactly the set of cone points.

\item The edges must be geodesic arcs of length less than $\pi$.

\item The interior face angles at each cone point must have measure less than $\pi$. 

\end{enumerate}
The first condition is clear from the construction of spherical polars.  The second condition is implied by the convexity of the hyperbolic polyhedron as edge lengths are equal to complements of interior dihedral angle measures.  The third condition is again a consequence of convexity, but convexity of the faces of a hyperbolic polyhedron as the face angles of a cone point have measures which are complementary to the face angles of the polyhedron.  Note that $\mathcal{C}$ satisfies these three conditions.

Note that if $c$ is a cone point of $Q_t$ and if $c$ is neither $w_1$ nor $w_2$, then the only other cone points at distance less than $\pi$ to $c$ are exactly distance $\frac{\pi}{2}$ away.  By conditions (1) and (2), these are the only possible endpoints of edges connecting $c$.  

Suppose $c$ is not $v_i$ or $w_i$ for $i=1,2$.  If $b$ is a cone point at distance $\frac{\pi}{2}$ from $c$ and if $b$ and $c$ are not connected by an edge in $\mathcal{D}$, then there must exist a 2--cell in $\mathcal{D}$ whose face angle at $c$ has measure $\pi$ or larger.  This contradicts condition (3).  Therefore, in $\mathcal{D}$, $c$ must be connected to every other cone point distance $\frac{\pi}{2}$ away, just as in $\mathcal{C}$.  

Consider next $v_1$.  The above argument shows that any cone point of distance $\frac{\pi}{2}$ from $v_1$ which is not $w_1$ or $w_2$ must be connected to $v_1$ by an edge in $\mathcal{D}$.  If $v_1$ and $w_1$ are not connected by an edge in $\mathcal{D}$, then there must be a 2--cell in $\mathcal{D}$ whose face angle at $w_1$ has measure $\pi$, a contradiction to condition (3).  This argument works just as well to show that both $v_1$ and $v_2$ are connected to both $w_1$ and $w_2$ by edges in $\mathcal{D}$.  

So far, it has been shown that every edge in $\mathcal{C}$ with an endpoint other than $w_1$ or $w_2$ must also be in $\mathcal{D}$.  The only edge in $\mathcal{C}$ not accounted for is the one connecting $w_1$ and $w_2$.  This must be in $\mathcal{D}$ as well since if it were not, there would be a 2--cell in $\mathcal{D}$ containing a face angle of measure $\pi$ at both $w_1$ and $w_2$ --- a contradiction.  Therefore, this edge is in $\mathcal{D}$ and so the edges of $\mathcal{C}$ are a subset of the edges of $\mathcal{D}$.

Given conditions (1) and (2), the only other possible edges in $\mathcal{D}$ are those edges not in $\mathcal{C}$ which connect $w_1$ (or $w_2$) to a cone point connected to $w_2$ (or $w_1$).  However, it is rather clear that such an edge would intersect an existing edge in $\mathcal{D}$ in an interior point which is a contradiction to condition (1).

Therefore $\mathcal{C} = \mathcal{D}$ and this proves Theorem \ref{P_t}.
\end{proof}

The set of polyhedra $P_t$, $t \in [0,1)$ form a 1--parameter family of hyperbolic polyhedra of fixed combinatorial type.  Therefore, the following classical result of Schl\"afli is applicable (see \cite{AV} for a proof).  It is often referred to as {\it Schl\"afli's Differential Formula}.

\begin{theorem} \label{Schlafli} If $P_t$ is a 1--parameter family of polyhedra in $\HH^n$, $n \geq 2$, then the derivative of the volume $\vol$ of $P_t$ is given by:
$$\frac{d\vol(P_t)}{dt} = \frac{-1}{n-1} \sum_F \vol(F) \frac{d\theta_F}{dt}$$
where the sum is taken over all codimension 2 faces of $P_t$ and $\theta_F$ is the measure of the dihedral angle of $P_t$ at the face $F$.  
\end{theorem}

This formula implies that as the dihedral angle measure of a hyperbolic polyhedron in any dimension increases, the volume decreases.  For the family $P_t$ constructed above, the dihedral angle along the very good edge $e$ is increasing from $\frac{\pi}{2}$ to $\pi$.  Therefore:

\begin{theorem} \label{edgesurgeryvoldecreasing} If $P_1$ is a right-angled hyperbolic polyhedron gotten by edge surgery along some very good edge of $P_0$, another right-angled hyperbolic polyhedron, then 
$$\vol(P_0) > \vol(P_1).$$ \qed
\end{theorem}

\section{Conclusion}
\setcounter{theorem}{0}

Let $P$ be a right-angled hyperbolic polyhedron.  Then Theorem \ref{vgcomp} implies that either $P$ is a L\"obell polyhedron, it is decomposable, or it has a very good edge along which edge surgery can be performed.  A similar trichotomy holds for the polyhedron or polyhedra which result after applying decomposition or edge surgery.  Note that each of these operations reduce the average number of faces of the polyhedra and so a process of repeated application of them must terminate in a set of L\"obell polyhedra after a finite number of steps.  Furthermore, by Theorems \ref{decompvoldecresing} and \ref{edgesurgeryvoldecreasing} the total volume of polyhedra does not increase at each step.  Therefore, the following result, the main result of this article, follows:

\begin{theorem} \label{organization} Let $P_0$ be a compact right-angled hyperbolic polyhedron.  Then there exists a sequence of disjoint unions of right-angled hyperbolic polyhedra $P_1, P_2, \dots, P_k$ such that for $i=1,\dots, k$, $P_i$ is gotten from $P_{i-1}$ by either a decomposition or edge surgery, and $P_k$ is a set of L\"obell polyhedra. Furthermore, 
$$\vol(P_0) \geq \vol(P_1) \geq \vol(P_2) \geq \dots \geq \vol(P_k).$$ \qed
\end{theorem}

With a blackbox or oracle which is able to compute volumes of right-angled polyhedra, this result would enable one to completely order the volumes of such objects.  In particular, the following result is a simple corollary:

\begin{corollary} \label{smallestvol} The compact right-angled hyperbolic polyhedron of smallest volume is $L(5)$ (a dodecahedron) and the second smallest is $L(6)$.
\end{corollary}

\begin{proof}  By Theorem \ref{lobellincreasing}, the smallest L\"obell polyhedron is $L(5)$ while the second smallest is $L(6)$.  Note that there is no polyhedron which results in $L(5)$ when an edge surgery is performed as $L(5)$ has only pentagonal faces.  Furthermore, it is obvious that any composition of L\"obell polyhedra will have volume larger than that of $L(5)$, and so the first result follows.

The only possible polyhedron besides $L(5)$ whose volume is possibly not larger than that of $L(6)$ is a composition of two copies of $L(5)$.  However, the volume of such an object is $2\vol(L(5)) = 8.612... > \vol(L(6)) = 6.023...$.  This proves the second result.
\end{proof}

\bibliographystyle{halpha}
\bibliography{ref}

\end{document}